\documentclass[a4paper,11pt]{amsart}
\usepackage{amsfonts}
\usepackage{amssymb}
\usepackage[utf8]{inputenc}
\usepackage{amsmath}
\usepackage{pdflscape}
\usepackage{float}
\usepackage{easyReview}
\usepackage{graphicx}
\usepackage[colorlinks=true, linkcolor=red, citecolor=blue]{hyperref}
\usepackage[]{epsfig}
\usepackage[]{pstricks}
\usepackage{tikz}

\newtheorem{theorem}{Theorem}[section]
\newtheorem{proposition}[theorem]{Proposition}
\newtheorem{lemma}[theorem]{Lemma}

\newtheorem{assumption}[theorem]{Assumption}

\theoremstyle{remark}

\usepackage{mathrsfs}

\numberwithin{equation}{section}
\makeatletter
\@namedef{subjclassname@2010}{\textup{2020} Mathematics Subject Classification}
\makeatother

\begin{document}
	
	\pagenumbering{arabic}	
	\title[Delayed K-KP-II system]{
Stabilization of the Kawahara-Kadomtsev-Petviashvili equation with time-delayed feedback
}
	\author[Capistrano-Filho]{Roberto de A. Capistrano-Filho*}
	\address{Departamento de Matem\'atica,  Universidade Federal de Pernambuco (UFPE), 50740-545, Recife (PE), Brazil.}
	\email{roberto.capistranofilho@ufpe.br}
	
	\author[Gonzalez Martinez]{Victor H. Gonzalez Martinez}
	\email{victor.martinez@ufpe.br}
		\author[Muñoz]{Juan Ricardo Muñoz}
			\email{juan.ricardo@ufpe.br}
	\thanks{*Corresponding author: roberto.capistranofilho@ufpe.br}
	
	\begin{abstract}
Results of stabilization for the higher order of the Kadomtsev-Petviashvili equation are presented in this manuscript.  Precisely, we prove with two different approaches that under the presence of a damping mechanism and an internal delay term (anti-damping) the solutions of the Kawahara-Kadomtsev-Petviashvili equation are locally and globally exponentially stable.  The main novelty is that we present the optimal constant, as well as the minimal time,  that ensures that the energy associated with this system goes to zero exponentially.
	\end{abstract}

	\maketitle
	
\section{Introduction}
\label{Sec0}
In the last years, properties of the asymptotic models for water waves have been extensively studied to understand the full water wave system\footnote{ See for instance  \cite{BLS, Lannes} and references therein, for a rigorous justification of various asymptotic models for surface and internal waves.}.  As well know,  we can formulate the waves as a free boundary problem of the incompressible, irrotational Euler equation in an appropriate non-dimensional form. Some physical conditions give us the so-called long waves or shallow water waves.  For example,  in one spatial dimensional case the so-called Kawahara equation which is an equation derived by Hasimoto and Kawahara in \cite{Hasimoto1970,Kawahara} that takes the form 
 \begin{equation}\label{eq:Kawahara} 
 \pm2 u_t + 3uu_x - \nu u_{xxx} +\frac{1}{45}u_{xxxxx} = 0.
   \end{equation} 

If we look for two spatial dimensional,  wave phenomena that exhibit weak transversality and weak nonlinearity are modeled by the Kadomtsev-Petviashvili (KP) equation
 \begin{equation}\label{eq:KP0a} 
 u_t+\alpha u_{xxx} +\gamma{\partial_x^{-1}} u_{yy}+uu_x = 0,
  \end{equation} 
  where $u=u(x,y,t)$ and $\alpha, \beta, \gamma$ are constants it was introduced by Kadomtsev and Petviashvili (see~\cite{kp70}) in 1970.  In 1993, Karpman included the  higher-order dispersion in \eqref{eq:KP0a} leads to a fifth-order generalization of the KP equation \cite{Karpman}
   \begin{equation}\label{eq:KP0} 
 u_t+\alpha u_{xxx} + \beta u_{xxxxx}+\gamma{\partial_x^{-1}} u_{yy}+uu_x = 0,
  \end{equation} 
  which will be called the Kawahara-Kadomtsev-Petviashvili equation (K-KP).  Note that, by scaling transformations on the variables $x, t$, and $u$, the coefficients in equation \eqref{eq:KP0} can be set to $\alpha>0, \beta<0, \gamma^2=1$. For the sequel, we consider this scaled form of the equation: 
   \begin{equation}\label{eq:KP0b} 
    u_t+u u_x+\alpha u_{x x x}+\beta u_{x x x x x}+\gamma \partial_x^{-1} u_{y y}=0, \quad \gamma=\pm 1 . 
      \end{equation} 
 When $\gamma=-1$ we will refer to the case as K-KP-I and for  $\gamma=1$ as K-KP II, respectively. This is motivated in analogy with the usual terminology for the KP equation, which distinguishes the two cases for the sign of the ratio of the highest derivative terms in $x$ and $y$, that is,  focusing and defocusing cases, respectively.

It is important to point out that there are several physical applications in modeling long water waves in a shallow water regime with a strong dispersion represented by systems \eqref{eq:Kawahara}--\eqref{eq:KP0b}.  We can cite at least two of them, the first one is to describe both the wave speed and the wave amplitude \cite{Haragus}, and the second one is modeling plasma waves with strong dispersion \cite{Kawahara}.

\subsection{Problem setting}	
There is an important advance in control theory to understand how the damping mechanism acts in the energy of systems governed by a partial differential equation.  In particular, exponential stability for dispersive equations related to water waves posed on bounded domains has been intensively studied.  For example, it is well known that the KdV equation~\cite{Zuazua2002}, Boussinesq system of KdV-KdV type~\cite{Pazoto2008}, Kawahara equation~\cite{Araruna2012} and others are exponentially stable using the Compactness-Uniqueness developed by J.L. Lions~\cite{Lions1988}.  Other results as obtained in~\cite{Cerpa2021} and in~\cite{Capistrano2018} are obtained using Urquiza's and Backstepping approach. All these results use damping mechanisms in the equation or the boundary as a control.

Recently,  in \cite{boumediene, martinez2022}, the authors obtained exponential decay for a fifth-order KdV type equation via the Compactness-Uniqueness argument and  Lyapunov approach.  Additionally to that,  in~\cite{Panthee2011} and~\cite{ailton2021},  exponential decay for the KP-II and K-KP-II were shown\footnote{See also the reference therein for stabilization of KP-II and K-KP-II.}.  In both works,  the authors can prove regularity and well-posedness for these equations and show that the energy associated with this equation decays exponentially in the presence of a damping term acting in the equation.

As we can see in these articles there is interest in the mathematical context in the study of the asymptotic behavior of the solution of the equation  \eqref{eq:KP0b}.  Additionally, as pointed out, the model under consideration in this article has importance in the context of the dispersive equation as well as, physical motivation.  So,  motived by \cite{boumediene, martinez2022,ailton2021,Panthee2011} we will analyze the qualitative properties of the initial-boundary value problem for the K-KP-II equation posed on a bounded domain $\Omega=(0,L)\times(0,L)\subset \mathbb{R}^2$ with localized damping and delay terms
\begin{equation}\label{eq:KP}
\begin{cases}
	\begin{aligned} 
		&\partial_tu(x,y,t) + \alpha \partial^3_xu(x,y,t) + \beta \partial^5_xu(x,y,t)\\
		&+ \gamma\partial_x^{-1}\partial^2_y u(x,y,t) + \frac{1}{2}\partial_x(u^2(x,y,t))\\
		&+ a(x,y)u(x,y,t) + b(x,y)u(x,y,t-h)=0,
	\end{aligned} & (x,y)\in \Omega,\ t>0. \\
	u(0,y,t)=u(L,y,t)=0 ,& y\in(0,L),\ t\in(0,T), \\
    \partial_xu(L,y,t) = \partial_xu(0,y,t)=\partial^2_xu(L,y,t)=0,& y\in(0,L),\ t\in(0,T), \\
	u(x,L,t)=u(x,0,t)=0,& x\in(0,L),\ t\in(0,T),\\
	u(x,y,0)=u_0(x,y), \quad
	u(x,y,t)= z_0(x,y,t), & (x,y)\in\Omega,\ t\in(-h,0).
\end{cases}
\end{equation}
Here $h>0$ is the time delay,   $\alpha>0$, $\gamma>0$ and $\beta<0$ are real constants.  Additionally,  define the operator $\partial_x^{-1}:=\partial_x^{-1} \varphi(x,y,t)= \psi(x,y,t)$ such that $\psi(L,y,t)=0$ and $\partial_x\psi(x,y,t)=\varphi(x,y,t)$\footnote{It can be shown that the definition of operator $\partial_x^{-1}$ is equivalent to $\partial_x^{-1} u(x,y,t) = \int_x^L u(s,y,t)\,ds$.} and,  for our purpose,  let us consider the following assumption.

	\begin{assumption}\label{A1}
		The real functions $a\left(x,y\right)$ and $b\left(x,y\right)$ are nonnegative belonging to $L^\infty(\Omega)$. Moreover, $a(x,y) \geq a_0>0$ is almost everywhere in a nonempty open subset $\omega \subset \Omega$.
	\end{assumption}

Our propose here is to present, for the first time, the K-KP-II system not with only a damping mechanism $a(x,y)u$, which plays the role of a feedback-damping mechanism (see e.g.  \cite{ailton2021}), but also with an anti-damping, that is, some feedback such that our system does not have decreasing energy.  In this context,  we would like to prove that the energy associated with the solutions of the system \eqref{eq:KP}
\begin{equation}\label{eq:KPen}
\begin{split}
E_{u}(t) =&\frac{1}{2} \int_0^L\int_0^L u^2(x,y,t) \,dx\,dy \\&+\frac{h}{2} \int_0^L\int_0^L\int_0^1 b(x,y)u^2(x,y,t-\rho h) \,d\rho\,dx\,dy.
\end{split}
\end{equation}
decays exponentially.  Precisely, we want to give an answer to the following question:

\vspace{0.2cm}

\textit{Does $ E_u(t) \rightarrow 0$ as $ t \rightarrow \infty$? If it is the case, can we give the decay rate?}

\subsection{Notation and main results}
Before presenting answers to this question, let us introduce the functional space that will be necessary for our analysis.  Given $\Omega \subset \mathbb{R}^2$ let us define $X^k(\Omega)$ to be the Sobolev space
\begin{equation}\label{eq:Xk}
X^k(\Omega):=
\left\{
\begin{array}
[c]{l}
\varphi \in H^k(\Omega) \colon {\partial_x^{-1}} \varphi(x,y)= \psi(x,y)\in H^k(\Omega)\text{ such that, }\\
\psi(L,y)=0\text{ and }\partial_x\psi(x,y)=\varphi(x,y)
\end{array}
\right\}
\end{equation} 
endowed with the norm
$
\left\lVert \varphi \right\rVert_{X^k(\Omega)}^2 = \left\lVert \varphi \right\rVert_{H^k(\Omega)}^2 + \left\lVert {\partial_x^{-1}} \varphi \right\rVert_{H^k(\Omega)}^2.
$
We also define the normed space $H_x^k(\Omega)$,
\begin{equation}\label{eq:Hxk}
H_x^k(\Omega):=
	\left\lbrace%
	 	\varphi \colon \partial_x^j \varphi \in L^2(\Omega),\ \text{for } 0\leq j \leq k
	 \right\rbrace 
\end{equation}
with the norm $\left\lVert \varphi \right\rVert_{H_x^k(\Omega)}^2 =\sum_{j=0}^k \left\lVert \partial_x^j \varphi\right\rVert_{L^2(\Omega)}^2$ and the space
\begin{equation}\label{eq:Xxk}
X_{x}^{k}(\Omega):=\left\{
\begin{array}
[c]{l}
\varphi \in H_x^k(\Omega) \colon {\partial_x^{-1}} \varphi(x,y)= \psi(x,y)\in H_x^k(\Omega)\text{ such that }\\
\psi(L,y)=0\text{ and }\partial_x\psi(x,y)=\varphi(x,y)
\end{array}
\right\}
\end{equation}
with $\left\lVert \varphi \right\rVert_{X_x^k(\Omega)}^2 = \left\lVert \varphi \right\rVert_{H_x^k(\Omega)}^2 + \left\lVert {\partial_x^{-1}} \varphi \right\rVert_{H_x^k(\Omega)}^2.$
Finally,   $H_{x0}^k(\Omega)$ will denote the closure of $C_0^\infty(\Omega)$ in $H_x^k(\Omega)$. 

The next result will be used repeatedly throughout the article:

\begin{theorem}[{\cite[Theorem~15.7]{besov79}}]\label{thm:Besov}
Let $\beta$ and $\alpha^{(j)}$, for $j = 1,\dots, N$, denote $n$- dimensional multi-indices with non-negative-integer-valued components. Suppose that $1 < p^{(j)} < \infty$, $1 < q < \infty$, $0 < \mu_j < 1$ with
$$ \sum_{ j = 1 }^N \mu_j = 1, \quad
\frac{1}{q} \leq \sum_{j = 1}^N \frac{\mu_j}{p^{(j)}},\quad \text{and} \quad \beta - \frac{1}{q} = \sum_{j=1}^N \mu_j\left(\alpha^{(j)} - \frac{1}{p^{(j)}}\right).$$
Then, for $f(x)\in C_0^\infty(\mathbb{R}^n)$,
\begin{equation*}
\left\lVert
	D^\beta f
\right\rVert_{q} 
\leq C \prod_{j=1}^N 
\left\lVert 
	D^{\alpha^{(j)}}f
\right\rVert_{p^{(j)}}^{\mu_j}.
\end{equation*}
Where, for non-negative multi-index  $\beta = (\beta_1,\dots, \beta_N)$  we denote $D^\beta$ by 
$D^\beta = D_{x_1}^{\beta_1}\dots D_{x_n}^{\beta_n}$
and 
$ D_{x_i}^{\beta_i} = \frac{\partial^{\beta_i}}{\partial x_i^{k_i}}$
\end{theorem}


The first result of the manuscript ensures that without a restrictive assumption on the length $L$ of the domain and with the weight of the delayed feedback small enough the energy \eqref{eq:KPen} associated with the solution of the system \eqref{eq:KP} are locally stable.
\begin{theorem}[Optimal local stabilization]\label{th:KPes}
Assume that the functions $a(x,y), b(x,y)$ satisfy the conditions given in Assumption \ref{A1}. Let $L>0$, $\xi>1$, $0<\mu<1$ and $T_0$ given by
\begin{equation}\label{time}
T_0 = \frac{1}{2\theta} \ln \left( \frac{2\xi \kappa }{\mu}\right) +1,
\end{equation}
with
$
	\theta=	\frac{3\alpha\eta}{(1+2\eta L)L^2}, 	\kappa  = 1+\max\left\lbrace{2\eta L, \frac{\sigma}{\xi}}\right\rbrace
$
and $\eta \in \left(0, \frac{\xi-1}{2L(1+2\xi)}\right)$ satisfying 
$$
\frac{2\alpha \eta}{(2+2\eta L)L^2}=\frac{\sigma}{2h(\xi+\sigma)}
$$
where $ \sigma=\xi -1 -2L\eta(1+2\xi)$.  Let $T_{\min}>0$ given by
\begin{equation*}
T_{\min}:= -\frac{1}{\nu} 
\ln\left( \frac{\mu}{2} \right)
+\left(
	\frac{2\lVert b \rVert_\infty}{\nu}+1
\right)T_0, \hbox{ with } \nu=\frac{1}{T_0}\ln\left(\frac{1}{(\mu+\varepsilon)}\right).
\end{equation*}
Then, there exists $\delta>0$, $r>0$, $C>0$ and $\gamma$, depending on $T_{\min}, \xi, L, h$, such that if $\lVert b \rVert_\infty \leq \delta$, then for every
$(u_0,z_0) \in \mathcal{H} = L^2(\Omega)\times L^2(\Omega\times(0,1))$
satisfying $\lVert (u_0, z_0) \rVert_{\mathcal{H}}\leq r$,
the energy of the system~\eqref{eq:KP} satisfies
$$
E_u(t)\leq Ce^{-\gamma t}E_u(0), \hbox{ for all } t>T_{\min}.$$
\end{theorem}

Now on, following the ideas in \cite{martinez2022}, we obtain some stability properties about the next system, called $\mu_i$--system.  Note that if we choose $a(x,y)= \mu_1 a(x,y)$ and $b(x,y)=\mu_2 a(x,y)$ in~\eqref{eq:KP}, where $\mu_1$ and $\mu_2$ are real constants we obtain the system
\begin{equation}\label{eq:MU}
\begin{cases}
	\begin{aligned} 
		&\partial_tu(x,y,t) + \alpha \partial_{x}^3u(x,y,t) + \beta\partial_{x}^5u(x,y,t) \\ &+ \gamma\partial_x^{-1} \partial^2_yu(x,y,t)+ \frac{1}{2}\partial_x(u^2(x,y,t)) \\
&+ a(x,y)\left(\mu_1u(x,y,t) + \mu_2 u(x,y,t-h)\right)=0,
	\end{aligned}& (x,y,t)\in \Omega\times\mathbb{R}^+\\
	u(0,y,t)=u(L,y,t)=0,& y\in(0,L), \ t\in(0,T), \\
	\partial_x u(L,y,t) = \partial_xu(0,y,t)=\partial^2_xu(L,y,t)=0,&y\in(0,L), \ t\in(0,T), \\
	u(x,L,t)=u(x,0,t)=0,& x\in(0,L), \ t\in(0,T),\\
	u(x,y,0)=u_0(x,y),\quad
	u(x,y,t)= z_0(x,y,t), & (x,y)\in \Omega, \ t\in(-h,0).
\end{cases}
\end{equation}
Here, $\mu_1>\mu_2$ are positive real number and $a(x,y)$ satisfies Assumption \ref{A1}.  We define the total energy associated to~\eqref{eq:MU}
\begin{equation}\label{eq:MUen}
\begin{split}
E_u(t)=&\frac{1}{2}\int_0^L\int_0^L  u^2(x,y,t)\,dx\,dy \\&+\frac{\xi}{2}\int_0^L\int_0^L  \int_0^1 a(x,y) u^2(x,y,t-\rho h)\,d\rho\,dx\,dy,
\end{split}
\end{equation}
where $\xi>0$  satisfies
\begin{equation}\label{eq:MUcond}
h\mu_2<\xi<h(2\mu_1-\mu_2).
\end{equation}
Note that the derivative of the energy \eqref{eq:MUen} satisfies 
	\begin{equation}\label{eq:MUenD}
\begin{aligned}
\frac{d}{dt} E_u(t) \leq&  -C%
	\left(
		\int_0^L \partial^2_xu(0,y,t)^2\,dy+ \int_0^L
			\left(%
				{\partial_x^{-1}} \partial_yu(0,y,t)%
			\right)^2\,dy
	\right. \\
&	\left.
		+\int_0^L\int_0^L  a(x,y)u^2(x,y,t-h)\,dx\,dy%
	\right)
\end{aligned}
\end{equation}
for $C:= C\left(\mu_1,\mu_2,\xi,h\right)\geq 0$. This indicates that the function $a(x,y)$ plays the role of a feedback-damping mechanism, at least for the linearized system. Therefore,  for the system \eqref{eq:MU} we split the behavior of the solutions into two parts.  Employing Lyapunov's method, it can be deduced that the energy $E_u(t)$ goes exponentially to zero as $t \rightarrow \infty$, however, the initial data needs to be sufficiently small in this case.  Precisely,  the second local result, can be read as follows:
\begin{theorem}[Local stabilization]\label{th:MUlyes}
Let $L>0$.  Assume that $a(x,y)\in L^\infty(\Omega)$ is a non-negative function, that  relation \eqref{eq:MUcond} holds and $\beta< -\frac{1}{30}$.
Then, there exists 
\begin{equation*}{
0<r< \frac{\sqrt[4]{216\alpha^3}}{CL^{\frac{5}{2}}}
}
\end{equation*}
such that for every  
$\left(u_0,z_0(\cdot,\cdot,-h(\cdot))\right)\in\mathcal{H}$
satisfying 
$\left\lVert{(u_0,z_0(\cdot,\cdot,-h(\cdot)))}\right\rVert_{\mathcal{H}} \leq r$,
the energy defined in~\eqref{eq:MUen} decays exponentially. More precisely, there exists two positives constants $\theta$ and $\kappa$ such that $E_u(t) \leq \kappa E_u(0)e^{-2\theta t}$ for all $t>0.$ Here,
\begin{equation*}
\theta
 < \min\left\lbrace
		\frac{\eta}{(1+2\eta L)L^2}
			\left[
				3\alpha 
				- \frac{1}{2} C^{\frac{4}{3}} r^{\frac{4}{3}} L^{\frac{10}{3}}
			\right],
			\frac{\xi\sigma}{2h(\xi+\sigma\xi)}
	\right\rbrace, \quad 
	\kappa  = 1+\max\lbrace{2\eta L, \sigma}\rbrace
\end{equation*}
and $\eta$ and $\sigma$ are positive constants such that
\begin{equation*}
\begin{split}
\sigma &
< \frac{2h}{\xi}\left(
		\mu_1-\frac{\mu_2}{2}-\frac{\xi}{2h}
	\right)\\
\eta &
  <\min\left\lbrace
		\frac{1}{2L\mu_2}
			\left[
				\frac{\xi}{h}-\mu_2
			\right],
		\frac{1}{2L\mu_1+L\mu_2}
			\left[
				\mu_1 - \frac{\mu_2}{2} 
				- \frac{\xi}{2h}(1+\sigma)
			\right]
	\right\rbrace.
\end{split}
\end{equation*}
\end{theorem}

The last result of the manuscript, still related to the system \eqref{eq:MU}, removes the hypothesis of the initial data being small.  To do that, we use the compactness-uniqueness argument due to J.-L. Lions \cite{Lions}, which reduces our problem to prove an \textit{observability inequality} for the nonlinear system \eqref{eq:MU} and removes the hypotheses that the initial data are small enough.

\begin{theorem}[Global stabilization]\label{th:MUes}
Let $a\in L^\infty(\Omega)$ satisfies Assumption \ref{A1}. Suppose that $\mu_1>\mu_2$ satisfies~\eqref{eq:MUcond}. Let $R>0$, then there exists $C=C(R)>0$ and $\nu=\nu(R)>0$ such that $E_u$, defined in~\eqref{eq:MUen} decays exponentially as $t$ tends to infinity,  when $\left\lVert(u_0,z_0)\right\rVert_{\mathcal{H}}\leq R$.
\end{theorem}

\subsection{Novelty and outline of the article}
We finish the introduction by highlighting some facts about our problem in comparison with the works previously mentioned, as well as, the organization of the manuscript. 

\begin{itemize}
\item[a.] Observe that the absence of drift term $u_x$, in comparison with Kawahara equation in~\cite{boumediene, martinez2022}, leads to get stabilization results without restriction in the length of the spatial domain. This term is not important in our analysis, the term only plays an important role in the problems where the control (damping or delay) is acting in the boundary condition\footnote{For details about this situation the authors suggest reference \cite{boumediene2}.}. 
\item[b.] As stated earlier, we introduce an anti-damping together with the damping mechanism to show that the energy of the system \eqref{eq:KP} decays exponentially. 
Compared with the known result  \cite{ailton2021}, the novelty of this paper is twofold: 
\begin{enumerate}
\item Our work gives the precise decay rate, see Theorems \ref{th:KPes} and \ref{th:MUlyes}.  
\item Lyapunov's method shows an optimal decay rate in terms of $\theta$ in Theorem \ref{th:KPes}.  Observe that the value of $\theta$ can be optimized as a function of $\eta$, that is, we can choose 
	\begin{equation}\label{alpha}
		\eta \in  \left(0,\frac{\xi-1}{2L(1+2\xi)}\right)
	\end{equation}
	such that the value of $\theta$ is the largest possible,  which implies that the decay rate $\theta$ thus obtained is the best one.  This can be seen defining the functions $ f, g :  \left[0,\frac{\xi-1}{2L(1+2\xi)}\right]  \longrightarrow   \mathbb{R}$ by
	\begin{equation*}
		f(\eta)= \frac{3\alpha\eta }{L^2(1+2\eta L)},
\quad
		g(\eta)= \frac{\xi-1-2L\eta(1+2\xi)}{2h(2\xi-1-2\eta L(1+2\xi))},
	\end{equation*}
and considering  $\gamma(\eta)=\min\{f(\eta),g(\eta)\}$. So,  the function $f$ is increasing in the interval $ \left[0,\frac{\xi-1}{2L(1+2\xi)}\right)$ while the function $g$ is decreasing in this same interval.  In fact, note that
	\begin{equation*}
		f(\eta)= \frac{3\alpha}{2L^3}\left(1-\frac{1}{1+2\eta L}\right)
	\end{equation*}
	and
	\begin{equation*}
		g(\eta)= \frac{1}{2h}-\left(\frac{\xi}{4hL(1+2\xi)}\right)\left(\frac{1}{\frac{\xi}{2L(1+2\xi)}+\frac{\xi-1}{2L(1+2\xi)}-\eta}\right).
	\end{equation*}
	If $-\frac{1}{2L}<\eta$, then
	\begin{equation*}
		f'(\eta)= \frac{3\alpha}{2L^3}\frac{2L}{(1+2L\eta)^2}>0.
	\end{equation*}
	In particular, $f'(\eta)>0$ when $$\eta \in \left[0,\frac{\xi-1}{2L(1+2\xi)}\right).$$ Analogously,
	\begin{equation*}
		g'(\eta)=-\left(\frac{\xi}{4hL(1+2\xi)}\right)\frac{1}{\left(\frac{\xi}{2L(1+2\xi)}+\frac{\xi-1}{2L(1+2\xi)}-\eta\right)^2}<0,
	\end{equation*}
	since $\xi>1$ and $\eta<\frac{\xi-1}{2L(1+2\xi)},$ showing our claim.  
	Now,  we claim that there exists only one point satisfying \eqref{alpha} such that $f(\eta)=g(\eta)$.  To show the existence of this point, it is sufficient to note that $f(0)=0$,  $g\left(\frac{\xi-1}{2L(1+2\xi)}\right)=0$ and 
	\begin{equation*}
		f\left(\frac{\xi-1}{2L(1+2\xi)}\right)=\frac{3\alpha}{2L^3} \left(\frac{3\xi-1}{3\xi}\right)>0, \quad 
		g(0)=\frac{1}{2h}\left(1-\frac{\xi}{2\xi-1}\right)>0.
	\end{equation*}
	The uniqueness follows from the fact that $f$ is increasing while $g$ is decreasing in this interval.
	\end{enumerate}

\item[c.] Taking into account the above information about $f$ and $g$, the maximum value of the function must be reached at the point $\eta$ satisfying \eqref{alpha}, where $f(\eta)=g(\eta)$.  The figure \ref{fig1} below shows, in a simple case,  what was said earlier to the functions $f$ and $g$ when we consider some values, for example, $L=1$, $\xi=2.3$, $\alpha=0.5$ and $h=1.5$:
\begin{figure}[!h]
\centering\includegraphics[width=3.5in]{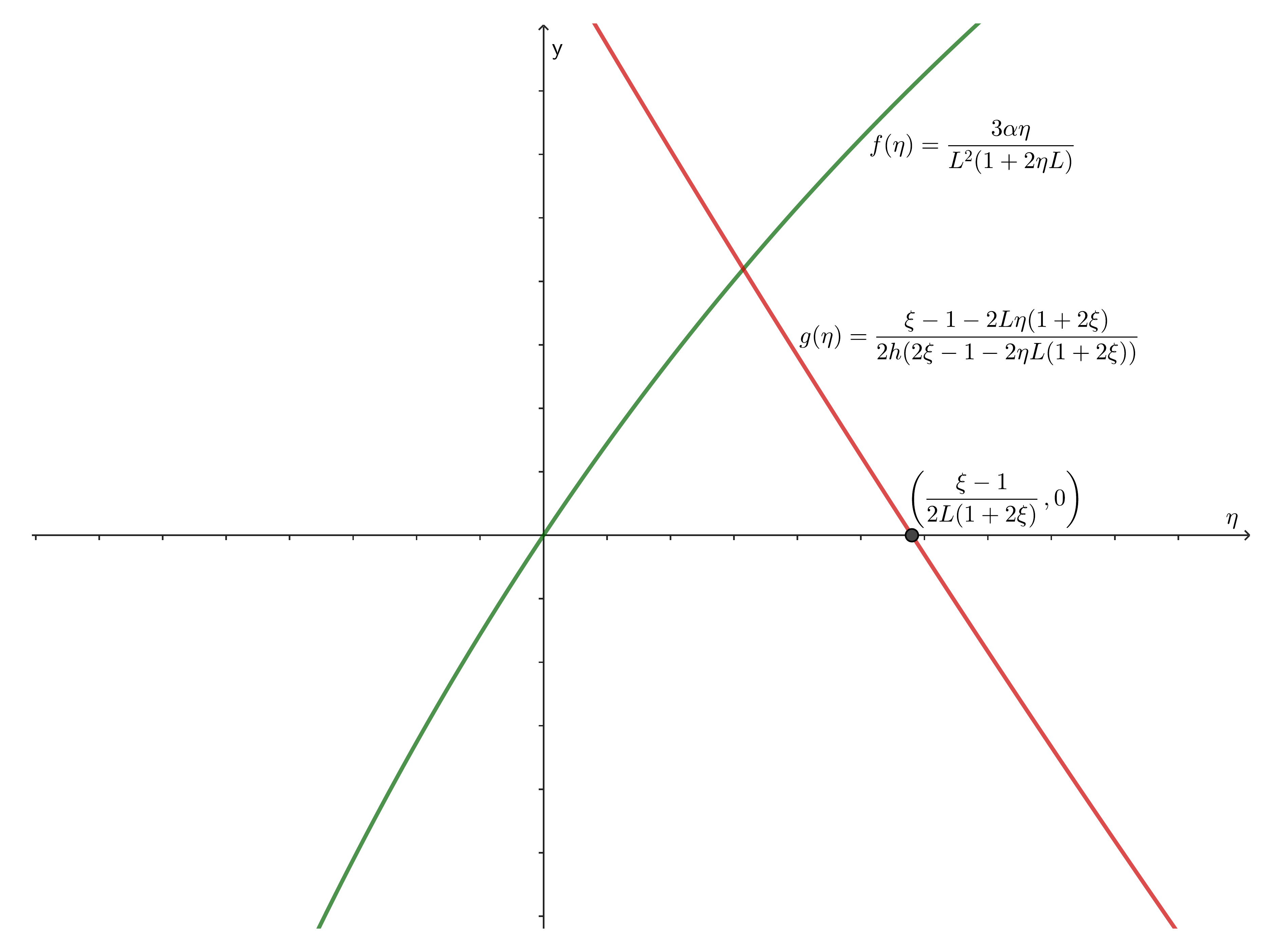}
\caption{Maximum of $\gamma(\eta)=\min\{f(\eta),g(\eta)\}$.}
	\label{fig1}
\end{figure}

	
\item[d.] Still concerning the Theorem \ref{th:KPes},  observe that we do not need to localize the solution of the transport equation in a small subset of $(0,L)$ as in \cite[Section 4]{Valein}. Moreover, we emphasize that we can take $a=0$ in Theorem \ref{th:KPes}. Finally, it is important to mention that we do not know if the time $T_{\min}$ is optimal. 
\item[e.] Aiming to present optimal decay results, note that for the nonlinear system we obtain one stabilization result with no restriction in the length of the spatial domain but carries a restriction in one parameter of the system, see Theorem \ref{th:MUlyes}.  Once again, it is possible to waive one of the conditions (either the restriction on $L$ or a restriction in one parameter of the system).  Observe that, using Theorem~\ref{thm:Besov} like as~\eqref{eq:MUkato2} below, we have
\begin{equation}\label{eq:MUgn}
\begin{aligned}
\int_0^L\int_0^L u^3(x,y,t)\,dx\,dy  
& \leq cL 
\left\lVert{u_{xx}}\right\rVert_{L^2(\Omega)}^{\frac{1}{2}}
\left\lVert{u}\right\rVert_{L^2(\Omega)}^{\frac{5}{2}} \\
& \leq 
{ \frac{1}{4}(CL)^{4}
	\left\lVert
		u_{xx}
	\right\rVert_{L^2(\Omega)}^2
+\frac{3}{4} r^{\frac{4}{3}}
	\left\lVert
		u
	\right\rVert_{L^2(\Omega)}^2.}
\end{aligned}
\end{equation}
This estimate allows obtaining, with an analogous argument another result for exponential stability without restriction in the parameter $\beta$ but with restriction in the length $L$ of the domain.  Thus,  in Theorem \ref{th:MUlyes},  we can remove the hypothesis over $\beta$, however, a hypothesis over $L$ is necessary. The result is the following one:
\begin{theorem}[Local stabilization-\textit{bis}]\label{th:MUlyes1}
Let  $0<L< \sqrt[4]{\frac{-30\beta}{C}}.$ 
Assume that $a(x,y)\in L^\infty(\Omega)$ is a non-negative function and that the relation \eqref{eq:MUcond} holds. 
Then, there exists  $0<r< \frac{\sqrt[4]{216\alpha^3}}{CL^{\frac{5}{2}}}$
such that for every  
$\left(u_0,z_0(\cdot,\cdot,-h(\cdot))\right)\in\mathcal{H}$
satisfying 
$\left\lVert{(u_0,z_0(\cdot,\cdot,-h(\cdot)))}\right\rVert_{\mathcal{H}} \leq r$,
the energy defined in~\eqref{eq:MUen} decays exponentially. More precisely, there exists two positives constants $\theta$ and $\kappa$ such that  $E_u(t) \leq \kappa E_u(0)e^{-2\theta t}$ for all $t>0,$ where $\theta$, 	 $\kappa$,  $\eta$ and $\sigma$ are positive constants defined as in Theorem \ref{th:MUlyes}. 
\end{theorem}
\item [f.]  The results obtained here can be easily adapted for the KP-II system \eqref{eq:KP0a} with or without the drift term $u_x$, extending the results of \cite{ailton2021} and \cite{Panthee2011}.
\end{itemize}

The work is organized as follows: 

-- Section \ref{Sec1} is devoted to proving the first, and optimal,  local stability result, that is, Theorem \ref{th:KPes}.  

-- In Section \ref{Sec2} we are able to prove the exponential stability, Theorem \ref{th:MUlyes},  for the energy associated with the $\mu_i$--system \eqref{eq:MU}.  

-- Additionally, to extend the local property to the global one,  in Section \ref{Sec2} we give the proof of Theorem \ref{th:MUes}. 

-- For the sake of completeness, we present in the Appendix, at the end of the work,  the well-posedness of the time-delayed K-KP-II system. 

\section{The damping-delayed system: Optimal local result}\label{Sec1}
This section deals with the behavior of the solutions associated with \eqref{eq:KP}. The first result ensures local stability considering the perturbed system.  After that, we are in a position to prove the first main result of the article, Theorem \ref{th:KPes}.

\subsection{Preliminaries} 
We are interested in analyzing the well-posedness of~\eqref{eq:KP} with total energy associated defined by \eqref{eq:KPen} that satisfies
\begin{equation}\label{eq:KPen6*}
\begin{aligned}
\frac{d}{dt} E_u(t) \leq & \int_0^L\int_0^L  b(x,y) u^2\,dx\,dy 
		+\frac{\beta}{2}\int_0^L u_{xx}^2(0,y,t)\,dy \\
&-\frac{\gamma}{2} \int_0^L 
		\left(%
			{\partial_x^{-1}} u_y(0,y,t)
		\right)^2\,dy		-\int_0^L\int_0^L  a(x,y) u^2(x,y,t)\,dx\,dy.
\end{aligned}
\end{equation} 
Which implies that the energy is not decreasing, in general, since the term $b(x,y)\geq0$. So, we consider the following perturbation system
\begin{equation}\label{eq:P}
\begin{cases}
	\begin{aligned} 
	&	\partial_tu(x,y,t) + \alpha\partial^3_x u(x,y,t) + \beta\partial^5_x u(x,y,t) \\&+\gamma\partial_x^{-1} \partial^2_yu(x,y,t)+ a(x,y)u(x,y,t) \\&+ b(x,y)(\xi u(x,y,t) + u(x,y,t-h))=f,
	\end{aligned} & (x,y)\in \Omega,\ t>0. \\
u(0,y,t)=u(L,y,t)=0,& y\in(0,L), \ t\in(0,T), \\
	\partial_x u(L,y,t) = \partial_xu(0,y,t)=\partial^2_xu(L,y,t)=0,&y\in(0,L), \ t\in(0,T), \\
	u(x,L,t)=u(x,0,t)=0,& x\in(0,L), \ t\in(0,T),\\
	u(x,y,0)=u_0(x,y),\quad
	u(x,y,t)= z_0(x,y,t), & (x,y)\in \Omega, \ t\in(-h,0),
\end{cases}
\end{equation}
with $f=-\frac{1}{2}\partial_x(u^2(x,y,t)$,  which is ``close'' to \eqref{eq:KP},  where $\xi$ a positive constant, and now the following energy associated with the perturbed system
\begin{equation}\label{eq:Pen}
\begin{split}
E_u(t) = &\frac{1}{2}\int_0^L\int_0^L u^2(x,y,t)\,dx\,dy\\&+\frac{\xi h}{2}\int_0^L\int_0^L\int_0^1 b(x,y) u^2(x,y,t-\rho h)\,d\rho\,dx\,dy,
\end{split}
\end{equation}
is decreasing. In fact,  note that 
\begin{equation*}
\begin{split}
&\frac{d}{dt} E_u(t)\leq \frac{\beta}{2}\int_0^L u_{xx}^2(0,y,t)\,dy
-\frac{\gamma}{2} \int_0^L 
	\left(%
	\partial_x^{-1} u_y(0,y,t)
	\right)^2\,dy\\
&-\int_0^L\int_0^L a(x,y) u^2(x,y,t)\,dx\,dy+\frac{1}{2}\int_0^L\int_0^L (b(x,y)-\xi b(x,y)) u^2(x,y,t) \,dx\,dy\\
&+\frac{1}{2}\int_0^L\int_0^L (b(x,y)-\xi b(x,y)) u^2(x,y,t-h)\,dx\,dy\leq 0,
\end{split}
\end{equation*}
for $\xi>1$.  
Note that the system \eqref{eq:P} can be written as a first-order system
\begin{equation}\label{eq:KPlinabs}
\begin{cases}
\dfrac{\partial}{\partial t}U(t)= AU(t), \\
U(0)=\left(u_0(x,y), z_0(x,y,-\rho h)\right).
\end{cases}
\end{equation}
Here $A = A_0+B$ with domain $D(A)=D(A_0)$,  $A_0$ is defined by 
$$
A_0(u, z) = 
	\left(%
		(-\alpha\partial_x^3 -\beta \partial_x^5-\gamma {\partial_x^{-1}} \partial^2_y-a(x,y))u-b(x,y)(\xi u+z(\cdot,\cdot, 1)),-h^{-1}\partial_{\rho}z
	\right)
$$
and the bounded operator $B$ is defined by $B(u,z) = \left(\xi b(x,y)u, 0\right),$ for all $(u,z)\in \mathcal{H}$.  Observe that system \eqref{eq:KPlinabs} has a classical solution (see Proposition \ref{pr:MUdis}).  

Consider $(e^{A_0t})_{t\geq0}$ the $C_0$--semigroup associated with $A_0$.  First, let us prove the exponential stability of the system~\eqref{eq:P}, with $f=0$, by using Lyapunov's approach. To do that, let us consider the following Lyapunov's functional 
\begin{equation*}
V(t) = E_u(t) + \eta V_1(t) + \sigma V_2(t),
\end{equation*}
where $\eta$ and $\sigma$ are suitable constants to be fixed later,
$E_u(t)$ is the energy defined by~\eqref{eq:Pen},
$V_1(t)$ is giving by
\begin{equation}\label{eq:MUv1}
V_1(t) = \int_0^L\int_0^L xu^2(x,y,t)\,dx\,dy 
\end{equation}
and $V_2(t)$ is defined by
\begin{equation}\label{eq:Pv2}
V_2(t) = \frac{h}{2} 
	\int_0^L\int_0^L \int_0^1
		(1-\rho) b(x,y) u^2(x,y,t-\rho h)\, d\rho\,dx\,dy.
\end{equation}
Note that $E_u(t)$ and $V(t)$ are equivalent in the following sense
\begin{equation}\label{eq:Pv2aa}
E(t) \leq V(t)
\leq
\left(
	1+\max\left\lbrace
		2\eta L, \frac{\sigma}{\xi}
		\right\rbrace
\right)E(t)
\end{equation}
Then, we have the next results for exponential stability to the system \eqref{eq:P} with $f=0$.
\begin{proposition}\label{pr:Pes1}
Let $L>0$.  Assume that $a(x,y)$ and $b(x,y)$ belonging to $L^\infty(\Omega)$ are nonnegative functions, $b(x,y) \geq b_0 >0$ in $\omega$ and $\xi>1$. Then for every $\left(u_0,z_0(\cdot,\cdot,-h(\cdot))\right)\in\mathcal{H}$
the energy defined in \eqref{eq:Pen} decays exponentially. More precisely, there exists two positives constants $\theta$ and $\kappa$ such that $E_u(t) \leq \kappa E_u(0)e^{-2\theta t}$ for all $t>0.$
Here,
\begin{equation*}
\theta
 < \min\left\lbrace
		\frac{3\alpha\eta}{(1+2\eta L)L^2},
			\frac{\sigma}{2h(\xi+\sigma)}
	\right\rbrace, \quad 
	\kappa  = 1+\max\left\lbrace{2\eta L, \frac{\sigma}{\xi}}\right\rbrace
\end{equation*}
and $\eta$ and $\sigma$ are positive constants such that $\sigma  = \xi -1 -2L\eta(1+2\xi)$ and $\eta  < \frac{\xi-1}{2L(1+2\xi)}.$
\end{proposition}
\begin{proof}
Consider $(u_0,z_0(\cdot,\cdot, - h(\cdot))) \in D(A_0)$. Let $u$ solution of the linear system associated with \eqref{eq:P}.  Differentiating \eqref{eq:MUv1} and using~\eqref{eq:P}$_1$, we obtain
\begin{equation*}
\begin{split}
\frac{d}{dt} V_1(t) 
= &-3\alpha 
	\int_0^L\int_0^L u_x^2(x,y,t)\,dx\,dy 
+5\beta
	\int_0^L\int_0^L u_{xx}^2(x,y,t)\,dx\,dy
\\&-\gamma
	\int_0^L\int_0^L
		\left({\partial_x^{-1}} u_y\right)^2\,dx\,dy-2\int_0^L\int_0^L xa(x,y) u^2(x,y,t)\,dx\,dy 
\\&-2
	\int_0^L\int_0^L x\xi b(x,y) u^2(x,y,t)\,dx\,dy \\&-2	\int_0^L\int_0^L x b(x,y) u(x,y,t) u(x,y,t-h)\,dx\,dy.
	\end{split}
\end{equation*}
Therefore, for $\theta>0$,  $\eta$ and $\sigma$ chosen as in the statement of proposition we have $\frac{d}{dt} V(t) + 2 \theta V(t) \leq 0,$ which is equivalent to
$$
E_u(t) \leq 
\left(
	1+\max\left\lbrace 2\eta L, \frac{\sigma}{\xi} \right\rbrace
\right) e^{-2\theta t} E(0),\quad \forall t>0,
$$
thanks to \eqref{eq:Pv2aa}.
\end{proof}

The next result shows that the energy \eqref{eq:KPen} associated with the system \eqref{eq:P} with appropriate source term $f$ decays exponentially. 
\begin{proposition}\label{pr:KPde}
Consider $a(x,y)$ and $b(x,y) \in L^\infty(\Omega)$ nonnegative functions, $b(x,y) \geq b_0 >0$ in $\omega$ and $\xi>1$.  So, there exists $\delta > 0$ such that if $\lVert \beta \rVert \leq \delta$ then, for every initial data $(u_0, z_0(\cdot,\cdot,-h(\cdot)) \in \mathcal{H}$ the energy of the system $E_u(t)$, defined in~\eqref{eq:KP} is exponentially stable.
\end{proposition}
\begin{proof}
Consider a function $v$ satisfying the system \eqref{eq:P} with $f=0$,  initial condition $v(x,y,0)=u_0(x,y)$, and $z^1(1)=u(x,y,t-h)$ where $z^1$ satisfies
\begin{equation}\label{eq:KPlinv}
\begin{split}
\begin{cases}
	hz_t^1(x,y,\rho,t) + z_\rho^1(x,y,\rho,t) = 0, & (x,y)\in \Omega,\ \rho\in (0,1),\ t>0\\
	z^1(x,y,0,t) = v(x,y,t), & (x,y)\in \Omega, t>0 \\
	z^1(x,y,\rho,0) = v(x,y,-\rho h)= z_0(x,y,-\rho h), & (x,y)\in\Omega,\ \rho\in(0,1).
\end{cases}
\end{split}
\end{equation}
and $w$ satisfying the source system associated with \eqref{eq:P} with $f=\xi b(x,y)v(x,y,t)$, initial condition $w(x,y,0)=0$ and $z^2(1)=u(x,y,t-h)$ where  $z^2$ satisfies
\begin{equation}\label{eq:KPlinw}
\begin{cases}
	hz_t^2(x,y,\rho,t)+z_\rho^2(x,y,\rho,t) = 0, & (x,y)\in \Omega,\ \rho \in (0,1),\ t>1 \\
	z^2(x,y,0,t) = w(x,y,t), & (x,y)\in \Omega,\ t>0 \\
	z^2(x,y,\rho, 0) = 0, & (x,y)\in \Omega,\ \rho \in(0,1) 
\end{cases}
\end{equation}
Define $u = v + w$ and $ z = z^1 + z^2$,  then $u$ satisfies the linear system associated with \eqref{eq:KP} where $z(1)=u(x,y,t-h)$with $z$ satisfying the equation \eqref{eq:MUtp}.

Now, fix $0<\mu <1$ and choose
\begin{equation*}
T_0 = \frac{1}{2\theta} \ln \left( \frac{2\xi \kappa }{\mu}\right) +1 \implies \kappa e^{-2\theta T_0} < \frac{\mu}{2\xi},
\end{equation*}
where $\eta, \sigma, \theta$ and $\kappa$ are given in the Proposition~\ref{pr:Pes1}. 
As $E_v(0) \leq \xi E_u(0)$, follows that
\begin{equation*}
E_v (T_0) \leq \kappa e^{-2\theta T_0} E_v(0) 
\leq \frac{\mu}{2\xi} E_v(0) \leq \frac{\mu}{2} E_u(0).
\end{equation*}
Observe that
\begin{equation*}
E_u(T_0)  \leq 2E_v(T_0) +
\left\lVert
\left(w(\cdot,\cdot, T_0), w(\cdot,\cdot, T_0-h(\cdot))\right)
\right\rVert_{\mathcal{H}}.
\end{equation*}
Since $A$ generates a $C_0$ semi-group we have that
\begin{equation*}
\begin{aligned}
\left\lVert
\left(w(\cdot,\cdot, T_0), w(\cdot,\cdot, T_0-h(\cdot))\right)
\right\rVert_{\mathcal{H}} 
& \leq 
\int_0^{T_0} e^{\frac{1+3\xi}{2}(T_0-s)}
	\left(
		\int_0^L 
			\left\lvert
				\xi b(x,y) v
			\right\rvert^2\,dx
	\right)^{\frac{1}{2}}\,ds \\
& \leq 
\sqrt{2\kappa}\xi\lVert b \rVert_{\infty} E_v(0)^{\frac{1}{2}}
	\int_0^{T_0}
		e^{\frac{1+3\xi}{2}(T_0-s)} e^{-\theta s}\,ds\\
&\leq 2\xi^2 \lVert b \rVert_{\infty}^2 e^{(3\xi+1)T_0}\kappa E_v(0),
\end{aligned}
\end{equation*}
thanks to the fact that
\begin{equation*}
\int_0^{T_0} e^{\frac{1+3\xi}{2}(T_0-s)}e^{-\theta s}\,ds
 = \frac{e^{\frac{1+3\xi}{2}T_0}-e^{-\theta T_0}}{\frac{1+3\xi}{2}+\theta}
\quad \text{and}\quad \frac{1+3\xi}{2}+\theta > 2.
\end{equation*}
For $\varepsilon>0$ such that $0 < \mu +\varepsilon < 1$ and
$
\left\lVert b \right\rVert_\infty
\leq \min \left\lbrace
	\frac{\sqrt{\varepsilon}}{\sqrt{\varepsilon^3 \kappa}
		e^{\frac{1+3\xi}{2}
			\left(
				\frac{1}{2\theta}\ln\left(\frac{2\xi\kappa}{\mu}\right)
				+2
			\right)}
		} , 1
\right\rbrace,
$
we obtain that,
\begin{equation*}
E_u(T_0) \leq \mu E_u(0) + 2\xi^3 \lVert b \rVert_{\infty}^2 
e^{(1+3\xi)T_0}\kappa E_u(0) < (\mu + \varepsilon)E_u(0).
\end{equation*}
 
 Finally, considering a boot-strap and induction arguments, for $T_0$ defined by \eqref{time}, we can construct another solution that satisfies the linear system associated with \eqref{eq:P} such that the following inequality holds
$
E_u(mT_0) \leq (\mu + \varepsilon)^m E_u(0),
$
for all $m\in \mathbb{N}$.  Picking $t>T_0$,  we note that there exists $m\in \mathbb{N}$ such that $t=mT_0+s$ with $0\leq s < T_0$, then
\begin{equation*}
E_u(t) \leq e^{(2\lVert b \rVert_\infty+\nu)T_0} e^{-\nu t} E_u(0),
\end{equation*}
where
\begin{equation}\label{eq:KPnu}
\nu = \frac{1}{T_0}\ln \left(\frac{1}{\mu+\varepsilon}\right),
\end{equation}
showing the result.
\end{proof}

\subsection{Proof of Theorem \ref{th:KPes}}With the previous result in hand,  in this section, we are going to prove a local stabilization result with an optimal decay rate. Using the same arguments in Section \ref{ss:MUnl} we have that~\eqref{eq:KP} is well-posed. Besides that we have by using Gronwall's inequality
\begin{equation*}
\lVert 
	(u(\cdot,\cdot,t), u(\cdot,\cdot, t-h(\cdot))) 
\rVert_{\mathcal{H}}^2
\leq
e^{2\xi \lVert b \rVert_\infty t}
\lVert 
	(u_0, z_0(\cdot,\cdot, -h(\cdot))) 
\rVert_{\mathcal{H}}^2
\end{equation*}
This implies directly that
\begin{equation*}
\lVert u \rVert_{C([0,T], L^2(\Omega))} 
\leq e^{\xi \lVert b \rVert_\infty T}
\lVert 
	(u_0, z_0(\cdot,\cdot, -h(\cdot))) 
\rVert_{\mathcal{H}}
\end{equation*}
and
\begin{equation*}
\lVert u \rVert_{L^2(0,T, L^2(\Omega))} 
\leq T^{\frac{1}{2}}e^{\xi \lVert b \rVert_\infty T}
\lVert 
	(u_0, z_0(\cdot,\cdot, -h(\cdot))) 
\rVert_{\mathcal{H}}
\end{equation*}
Now, multiplying the system \eqref{eq:KP} by
$xu(x,y,t)$, integrating by parts in $\Omega\times(0,T)$ we get 
\begin{equation*}
\begin{split}
&\frac{3\alpha}{2}
\int_0^T\int_0^L\int_0^L
	u_x^2(x,y,t)\,dx\,dy\,dt
-\frac{5\beta}{2}
\int_0^T\int_0^L\int_0^L
	u_{xx}^2(x,y,t)\,dx\,dy\,dt\\
& \leq
\left(
	\frac{L}{2} + L
		\left(
			\lVert a \rVert_\infty + \lVert b \rVert_\infty
		\right)Te^{2\xi\lVert b \rVert_\infty T}
\right)
\left\lVert
	(u_0,z_0(\cdot,\cdot, -h(\cdot)))
\right\rVert_{\mathcal{H}}^2 \\&+ 
\int_0^T\int_0^L\int_0^L 
	\lvert 
		u(x,y,t) 
	\rvert^3\,dx\,dy\,dt
	 \end{split}
\end{equation*}
From
\begin{equation}\label{eq:MUkato2}
\int_0^L\int_0^L u^3(x,y,t)\,dx\,dy
	\leq \frac{\varepsilon^4}{4} 
	\left\lVert{u}\right\rVert_{H_x^2(\Omega)}^2 
+	\frac{3}{4}
	\left(
		\frac{CL}{\varepsilon}
	\right)^{\frac{4}{3}}
	\left\lVert{u}\right\rVert_{L^2(\Omega)}^{\frac{10}{3}},
\end{equation}
and taking $E_u(0)\leq 1$, yields
\begin{equation*}
\lVert u \rVert_{B_H}^2 \leq \mathcal{\tilde{K}}
\left(
1+Te^{2\lVert b \rVert_\infty T} 
+ Te^{\frac{10}{3}\lVert b \rVert_\infty T}
+ e^{2\lVert b \rVert_\infty T}
\right) E_u(0),
\end{equation*}
where
\begin{equation*}
\mathcal{\tilde K}
:=
\frac{1}{\min\lbrace{1,3\alpha/2, -5\beta/2}\rbrace}
\left(
	\frac{L}{2} + L(\lVert a \rVert_\infty + \lVert b \rVert_\infty) 
	+\frac{1}{4}
	\left(
		\frac{cL}{\tilde{\delta}}
	\right)^{\frac{4}{3}}
\right)
\end{equation*}

Observe that, by definition, 
${\partial_x^{-1}} u(\cdot,\cdot, t) = \varphi(\cdot,\cdot,t) \in H_{x0}^2$ 
such that $\partial_x\varphi(\cdot,\cdot,t) = u(\cdot,\cdot,t)$.
Since $u\in H_{x0}^2$,  using Poincaré's inequality, we have that
\begin{equation*}
\left\lVert
	{\partial_x^{-1}} u(\cdot,\cdot,t) 
\right\rVert_{L^2(\Omega)}=
\left\lVert
	\varphi(\cdot,\cdot,t) 
\right\rVert_{L^2(\Omega)}\leq
L^2
\left\lVert
	\partial_x\varphi(\cdot,\cdot,t) 
\right\rVert_{L^2(\Omega)} = L^2
\left\lVert
	 u(\cdot,\cdot,t) 
\right\rVert_{L^2(\Omega)}.
\end{equation*}
Therefore,
\begin{equation*}
\lVert u \rVert_{\mathcal{B}_X}^2 \leq (1+L^2)
\mathcal{\tilde K}
\left(
1+Te^{2\lVert b \rVert_\infty T} 
+ Te^{\frac{10}{3}\lVert b \rVert_\infty T}
+ e^{2\lVert b \rVert_\infty T}
\right) E_u(0).
\end{equation*}

Let $(u_0,z_0(\cdot,\cdot,-h(\cdot)))$ be a initial data satisfying $\lVert (u_0,z_0(\cdot,\cdot,-h(\cdot)))\rVert_{\mathcal{H}}\leq r,$ where $r$ to be chosen later. The solution $u$ of \eqref{eq:KP} can be written as $u=u^1+u^2$ where $u^1$ is solution of the linear system associated with \eqref{eq:KP} considering the initial data  $u^1(x,y,0)=u_0(x,y)$ and $u^1(x,y,t) = z_0(x,y,t)$ and $u^2$ fulfills the nonlinear system  \eqref{eq:KP} with initial data $u^2(x,y,0)=0$ and $u^2(x,y,t) = 0$.

Fix $\mu \in (0,1)$, follows the same ideas introduced by \cite[Appendix A]{martinez2022}, there exists, $T_1>0$ such that
\begin{equation*}
e^{(2\lVert b \rVert_\infty + \nu)T_0-\nu T_1}< \frac{\eta}{2} 
\Longleftrightarrow
T_1 > -\frac{1}{\nu}\ln
\left(
	\frac{\eta}{2}
\right)
+\left(
	\frac{2\lVert b \rVert_\infty}{\nu} + 1
\right)T_0
\end{equation*}
with $\nu$ is defined by~\eqref{eq:KPnu} satisfiying $E_{u^1}(T_1) \leq \frac{\mu}{2} E_{u^1}(0).$
This implies together with \eqref{eq:MUkato2} that
\begin{equation*}
\begin{split}
E_u(T_1)
 \leq &\mu E_{u}(0)
+\left\lVert
(u^2(\cdot,\cdot,T_1), u^2(\cdot,\cdot,T_1-h(\cdot)))
\right\rVert_{\mathcal{H}}^2 \\
\leq&\mu E_u(0) + e^{(1+3\xi)T_1}
\left\lVert
uu_x
\right\rVert_{L^1(0,T_1, L^2(\Omega))}^2\\
\leq&\mu E_u(0)
+ e^{(1+3\xi)T_1}C_1^2C_2^2T^{\frac{1}{2}}
\lVert u \rVert_{\mathcal{B}_{X}}^4\\
\leq&(\mu+\mathcal{R})E_u(0),
\end{split}
\end{equation*}
where $$\mathcal{R}=e^{(1+3\xi)T_1}C_1^2C_2^2T_1^{\frac{1}{2}}
	(1+L^2)^2
\mathcal{\tilde K}^2
\left(
1+T_1e^{2\lVert b \rVert_\infty T_1} 
+ T_1e^{\frac{10}{3}\lVert b \rVert_\infty T_1}
+ e^{2\lVert b \rVert_\infty T_1}
\right)^2r.$$
Therefore, given $\varepsilon>0$ such that $\mu+\varepsilon < 1$, 
we take $r>0$ such that
\begin{equation*}
r < \frac{\varepsilon}{e^{(1+3\xi)T_1}C_1^2C_2^2T_1^{\frac{1}{2}}
	(1+L^2)^2
\mathcal{\tilde K}^2
\left(
1+T_1e^{2\lVert b \rVert_\infty T_1} 
+ T_1e^{\frac{10}{3}\lVert b \rVert_\infty T_1}
+ e^{2\lVert b \rVert_\infty T_1}
\right)^2}
\end{equation*}
to obtain $E_u(T_1) \leq (\mu+\varepsilon) E_u(0),$ with $\mu+\varepsilon < 1$. Using a prolongation argument, first for the time $2T_1$ and after for $mT_1$, the result is obtained. 
\qed


\section{\texorpdfstring{$\mu_i$}{}-system: Stability results}\label{Sec2} The main objective of this section is to prove the local and global exponential stability for the solutions of \eqref{eq:MU} using two different approaches. 

\subsection{Local  stabilization: Proof of Theorem \ref{th:MUlyes}}  Consider the Lyapunov's functional
$
V(t) = E_u(t) + \eta V_1(t) + \sigma V_2(t),
$
where $E_u(t)$ is defined by~\eqref{eq:MUen},  $V_1(t)$ defined by \eqref{eq:MUv1}
and
\begin{equation}\label{eq:MUv2}
V_2(t) = \frac{\xi}{2}
	\int_0^L\int_0^L\int_0^1
		(1-\rho)a(x,y)u^2(x,y,t-\rho h)\,d\rho\,dx\,dy.
\end{equation}
Using the same argument as in the proof of Proposition \ref{pr:Pes1} we see that 
\begin{equation}\label{eq:MUvv}
\begin{split}
&\frac{d}{dt} V(t) + 2\theta V(t) 
 \leq 
\left(
	\frac{\mu_2}{2}-\frac{\xi}{2h}+\eta L\mu_2
\right)
\int_0^L\int_0^L a(x,y) u^2(x,y,t-h)\,dx\,dy \\
& +\left(
	\theta\xi-\frac{\xi}{2h}\sigma+\theta\sigma\xi
\right)
\int_0^L\int_0^L\int_0^1 a(x,y) u^2(x,y,t-\rho h)\,d\rho\,dx\,dy \\
&+\left(
	\frac{\mu_2}{2}-\mu_1+\frac{\xi}{2h}+2\eta L \mu_1 + 
	\eta L \mu_2 + \frac{\xi}{2h}\sigma
\right) 
\int_0^L\int_0^L a(x,y) u^2(x,y,t)\,dx\,dy \\
& +(\theta+2\theta\eta L)
\int_0^L\int_0^L u^2(x,y,t)\,dx\,dy 
-3\alpha\eta 
\int_0^L\int_0^L u_x^2(x,y,t)\,dx\,dy\\
&+\frac{2}{3}\eta
\int_0^L\int_0^L u^3(x,y,t)\,dx\,dy + 5\beta \eta
\int_0^L \int_0^L u_{xx}^2(x,y,t)\,dx\,dy,
\end{split}
\end{equation}
for all $\theta>0$.  Note that,  thanks to Theorem \ref{thm:Besov} we have
\begin{equation*}
\int_0^L\int_0^L u^3(x,y,t)\,dx\,dy  
\leq 
{ \frac{1}{4}
	\left\lVert
		u_{xx}
	\right\rVert_{L^2(\Omega)}^2
+\frac{3}{4}(C L)^{\frac{4}{3}} r^{\frac{4}{3}}
	\left\lVert
		u
	\right\rVert_{L^2(\Omega)}^2.}
\end{equation*}
Putting this previous inequality in~\eqref{eq:MUvv},  and using Poincaré's inequality and \eqref{eq:MUgn},  we get
\begin{equation*}
\begin{aligned}
\frac{d}{dt} V(t) +& 2\theta V(t) \leq  \left(
		5\beta\eta + \frac{1}{6}\eta
	\right)\int_0^L\int_0^L u_{xx}^2(x,y,t)\,dx\,dy \\
	&+\left(
	\theta(1+2\eta L) L^2 
	+ \frac{1}{2}\eta C^{\frac{4}{3}} r^{\frac{4}{3}} L^{\frac{10}{3}}
	-3\alpha \eta
\right)
	\int_0^L\int_0^L u_x^2(x,y,t)\,dx\,dy.
\end{aligned}
\end{equation*}
Consequently, taking the previous constant as in the statement of the theorem we have that
\begin{equation}\label{Vn}
	V'(t)+2\gamma V(t) \leq 0.
\end{equation}
Finally, from the following relation
$E(t) \leq V(t) \leq \left(1+\max\left\lbrace{2\eta L,\sigma}\right\rbrace\right) E(t)$
 and \eqref{Vn}, we obtain
\begin{equation*}
E(t) \leq V(t) \leq e^{-2\theta t}V(0) 
\leq (1+\max\lbrace 2\eta L, \sigma \rbrace)e^{-2\sigma t}E(0),\quad \forall t>0,
\end{equation*}
and Theorem \ref{th:MUlyes} is proved. \qed

\subsection{Global stabilization:  Proof of Theorem \ref{th:MUes}}
As is classical in control theory, Theorem \ref{th:MUes} is a consequence of the following observability inequality
\begin{equation}\label{eq:MUobs}
\begin{aligned}
E_u(0)  \leq& \mathcal{C}
	\left(
		 \int_0^T\int_0^L \partial^2_xu(0,y,t)^2\,dy
		+\int_0^T\int_0^L({\partial_x^{-1}} \partial_yu(0,y,t))^2\,dy\,dt
	\right.\\
	&
	\left.%
		+\int_0^T\int_0^L\int_0^L a(x,y)(u^2(x,y,t-h)+u^2(x,y,t)\,dx\,dy\,dt%
	\right)
\end{aligned}
\end{equation}
Observe that using the same ideas of  \eqref{eq:MUr3}, we get
\begin{equation}\label{eq:MUobs1}
\begin{split}
&T\left\lVert{u_0}\right\rVert_{L^2(\Omega)}^2 
\leq \left\lVert{u}\right\rVert_{L^2(0,T,L^2(\Omega))}^2 
- \beta T \int_0^T\int_0^L \partial^2_xu(0,y,t)^2\,dy\,dt 
\\&+ \gamma T \int_0^T\int_0^L 
	\left(%
		{{\partial_x^{-1}} \partial_yu(0,y,t)}
	\right)^2\,dy\,dt\\&+T(2\mu_1+\mu_2) \int_0^T\int_0^L\int_0^L  a(x,y) u^2(x,y,t)\,dx\,dy\,dt \\
&+ T\int_0^T\int_0^L\int_0^L  a(x,y) \mu_2 u^2(x,y,t-h)\,dx\,dy\,dt
\end{split}
\end{equation}
Moreover,  multiplying~\eqref{eq:MUlin1}$_5$ by $a(x,y)\xi z(x,y,\rho,s)$, integrating in $\Omega\times(0,1)\times(0,T)$ and taking in account that $z(x,y,\rho,t)=u(x,y,t-\rho h)$ we obtain
\begin{equation}\label{eq:MUobs2}
\begin{split}
&\int_0^L\int_0^L\int_0^1 a(x,y)z^2(x,\rho,0)\,d\rho\,dx\,dy\\& \leq \frac{1}{hT}\int_0^T\int_0^L\int_0^L a(x,y) u^2(x,y,t)\,dx\,dy\,dt\\
&+ \left(%
	\frac{1}{Th}+\frac{1}{h}
	\right)
\int_0^T\int_0^L\int_0^L a(x,y) u^2(x,y,t-h)\,dx\,dy\,dt
\end{split}
\end{equation}
Gathering \eqref{eq:MUobs3} and \eqref{eq:MUobs1}, we see that to show \eqref{eq:MUobs} is sufficient to prove that for any $T$ and $R>0$, there exists $K:=K(R,T)>0$ such that
\begin{equation}\label{eq:MUobs3}
\begin{split}
&\left\lVert u \right\rVert_{L^2(0,T, L^2(0,L))}^2  \leq {K}
	\left(
		 \int_0^T\int_0^L \partial_{x}^2u(0,y,t)^2\,dy\right.\\
		&+\int_0^T\int_0^L({\partial_x^{-1}} \partial_yu(0,y,t))^2\,dy\,dt	
		+\int_0^T\int_0^L\int_0^L a(x,y)u^2\,dx\,dy\,dt\\
	&\left.
		+\int_0^T\int_0^L\int_0^L a(x,y)u^2(x,y,t-h)\,dx\,dy\,dt
	\right)
\end{split}
\end{equation}
holds for all solutions of~\eqref{eq:MU} with initial data $\left\lVert (u_0,z_0(\cdot,\cdot,-h(\cdot)))\right\rVert_{\mathcal{H}}\leq R$.

To prove it, let us argue by contradiction.  Suppose that~\eqref{eq:MUobs3} does not holds, then there exists a sequence $\left(u^n\right)_{n}\subset \mathcal{B}_X$ of solutions of~\eqref{eq:MU} with initial data $\left\lVert(u_0^n,z_0^n(\cdot,\cdot,-h(\cdot)))\right\rVert_{\mathcal{H}}\leq R$ such that
$
\lim_{n\to\infty} \frac{\left\lVert u^n \right\rVert_{L^2(0,T, L^2(\Omega))}^2}{B(u^n)}= +\infty
$
where
\begin{equation*}
\begin{split}
B(u^n) =& \int_0^T\int_0^L \lvert \partial^2_xu^n(0,y,t)\rvert^2\,dy+\int_0^T\int_0^L\lvert({\partial_x^{-1}}\partial_y u^n(0,y,t))\rvert^2\,dy\,dt\\
		&+\int_0^T\int_0^L\int_0^L a(x,y)\left(\lvert u^n(x,y,t)\rvert^2+\lvert u^n(x,y,t-h)\rvert^2\right)\,dx\,dy\,dt.
\end{split}
\end{equation*}
Let  $\lambda_n = \left\lVert u^n \right\rVert_{L^2(0,T,L^2(\Omega))}$ and $v^n(x,y,t) = 1/\lambda_n u^n(x,y,t)$,  then $v^n$ satisfies $\eqref{eq:MU}_1$ with the following boundary conditions
\begin{equation}\label{eq:MUvn}
\begin{cases}
	v^n(0,y,t)=v^n(L,y,t)=0,& y\in(0,L), t>0, \\
	\partial_xv^n(L,y,t) = \partial_xv^n(0,y,t)=\partial^2_xv^n(L,y,t)=0,& y\in(0,L), t>0, \\
	v^n(x,L,t)=v^n(x,0,t)=0,& x\in(0,L), t>0 \\
	v^n(x,y,0)=\frac{u_0}{\lambda_n}(x,y), \quad 
	v^n(x,y,t)=\frac{z_0}{\lambda_n}(x,y,t), & (x,y)\in\Omega,\ t\in(-h,0)\\
\left\lVert{v^n}\right\rVert_{L^2(0,T, L^2(\Omega))}^2=1
\end{cases}
\end{equation}
and $B(v^n)\to 0 $ as $n\to\infty$. Therefore, we have from~\eqref{eq:MUobs1} that
\begin{equation}
\left\lVert{v^n(\cdot,\cdot, t)}\right\rVert_{L^2(\Omega)}^2 \leq \frac{1}{T}\left\lVert{v^n}\right\rVert_{L^2(0,T, L^2(\Omega))}^2 + cB(v^n)
\end{equation}
which together with $\eqref{eq:MUvn}_6$ and $B(v^n)\to 0 $ gives that
$\left(v^n(\cdot,\cdot,0)\right)_n$ is bounded in $L^2(\Omega)$.  Additionally to that, the following inequality (see \eqref{eq:MUobs2}) 
\begin{equation*}
\begin{split}
 \int_{\Omega}\int_0^1& a(x,y)\frac{1}{\lambda_n^2}
	\left\lvert 
		z^n(x,\rho,0)
	\right\rvert^2\,d\rho\,dx\,dy 
\leq  \frac{1}{hT}\int_0^T\int_{\Omega} a(x,y) 
	\left\lvert 
		v^n(x,y,t)
	\right\rvert^2\,dx\,dy\,dt\\
&+	\left(%
		\frac{1}{hT}+\frac{1}{h}
	\right)
\int_0^T\int_{\Omega} a(x,y) 
	\left\lvert 
		v^n(x,y,t-h)
	\right\rvert^2\,dx\,dy\,dt 
	\end{split}
\end{equation*}
ensures that $\left(\sqrt{a(x,y)}v^n(\cdot,\cdot,-h(\cdot))\right)_n$ is bounded in $L^2(\Omega\times(0,1))$ and from \eqref{eq:MUr1}, $(\lambda_n)_n\subset \mathbb{R}$  is bounded.
On the other hand,  as a consequence of Proposition~\ref{pr:MUreg} we have that
$\left(v^n\right)_n$ is bounded in $L^2(0,T, H_{x}^2(\Omega))$.  Now,  using Theorem~\ref{thm:Besov}, we get
\begin{equation*}
\left\lVert v^nv_x^n \right\rVert_{L^2(0,T, L^1(\Omega))} 
\leq C^2
	\left\lVert%
		v^n
	\right\rVert_{L^\infty(0,T, L^2(\Omega))}^{\frac{3}{2}}
	\left\lVert%
		v^n
	\right\rVert_{L^2(0,T, H_x^2(\Omega))}
\end{equation*}
and $(v^n v_x^n)_n$ is bounded in $L^2(0,T,L^1(\Omega))$.  Defining $\partial_yv^n =\partial_x \varphi^n$,  and using once again Theorem~\ref{thm:Besov} we have
$
\left\lVert
	{\partial_x^{-1}} v_{yy}^n 
\right\rVert_{L^2(\Omega)} \leq C^2
\lVert v_x^n\rVert_{L^2(\Omega)} < \infty.
$
Consequently, using Cauchy-Schwarz inequality
\begin{equation*}
\left\lvert
	\left\langle
		{\partial_x^{-1}}v_{yy}^n , \xi
	\right\rangle_{H^{-3}(\Omega), H_0^3(\Omega)}
\right\rvert
 \leq 
	\left\lVert
		 \varphi_{y}^n
	\right\rVert_{L^2(\Omega)}
	\left\lVert
		 \xi
	\right\rVert_{L^2(\Omega)} 
 \leq C^2
	\left\lVert
		v_x^n
	\right\rVert_{L^2(\Omega)}
	\left\lVert
		 \xi
	\right\rVert_{L^2(\Omega)}. 
\end{equation*}
Observe that $(v^n)_n$ bounded in $L^2(0,T; H_x^2(\Omega))$ implies, in particular, that $(v_x^n)_n$ is bounded in $L^2(0,T, L^2(\Omega))$,  so
\begin{equation*}
\begin{split}
\left\lVert {\partial_x^{-1}} v_{yy}^n \right\rVert^2_{L^2(0,T; H^{-3}(\Omega))} 
			 \leq 
	{C}
		\int_0^T 
			\left\lVert
		v_{xx}^n
	\right\rVert_{L^2(\Omega)}
	\left\lVert
		v^n
	\right\rVert_{L^2(\Omega)}\,dt
	 \leq \frac{C}{2} \left\lVert v^n\right\rVert_{L^2(0,T,H_x^2(\Omega))},
\end{split}
\end{equation*}
where we used that $H_x^2(\Omega)\subset L^2(\Omega)$.  

Thus, the previous analysis ensures that
\begin{equation*}
\begin{aligned} 
		v_t^n(x,y,t) =&-\alpha v_{xxx}^n(x,y,t) + \beta v_{xxxxx}^n(x,y,t) + \gamma\partial_x^{-1} v_{yy}^n(x,y,t)\\
		&+ \lambda_nv^n(x,y,t)v_x^n(x,y,t)+ a(x,y)\left(\mu_1v^n(x,y,t) + \mu_2 v^n(x,y,t-h)\right),
	\end{aligned}
\end{equation*}
is bounded in $L^2(0,T, H^{-3}(\Omega))$, which together with a  classical compactness results\footnote{See \cite{Simon}.}, give us the existence of a sequence $(v^n)_n$ relatively compact in $L^2(0,T, L^2(\Omega))$, that is, there exists a subsequence, still denoted $(v^n)_n$,
\begin{equation}\label{sim}
v_n \rightarrow v \hbox{ in } L^2(0,T,L^2(\Omega))
\end{equation}
with
$\left\lVert v \right\rVert_{L^2(0,T, L^2(\Omega))}=1.$ 

Finally, from  weak lower semicontinuity of convex functional, we obtain
\begin{equation}\label{v_st}
v(x,y,t)=0 \in \omega\times(0,T)\text{ and } \partial^2_xv(0,y,t)=0\text{ in } (0,L)\times(0,T).
\end{equation}
Since $(\lambda_n)_n$ is bounded, we can extract a subsequence denoted $(\lambda_n)_n$ which converges to $\lambda\geq 0$.

We claim that ${\partial_x^{-1}} \partial^2_yv^n \rightarrow {\partial_x^{-1}} \partial^2_yv$ in $L^2(0,T,H^{-2}(\Omega))$. In fact, from definition of $\mathcal{B}_X$ we have ${\partial_x^{-1}} v^n = \varphi^n$ where
$\partial_x\varphi^n = v^n$, $v^n(\cdot,\cdot,t)\in H_{x0}^1(\Omega)$ and $\varphi^n(\cdot,\cdot,t)\in H_{x0}^1(\Omega)$. Since  $\partial_x^{-1}\partial^2_yv^n =\partial^2_y \varphi^n$ we obtain
\begin{equation*}
	\begin{split}
		&\left\lVert
		{\partial_x^{-1}}v_{yy}^n(\cdot,\cdot,t) - 
		{\partial_x^{-1}} v_{yy}(\cdot,\cdot,t)
		\right\rVert_{H^{-2}(\Omega)} =
		\left\lVert
		\varphi_{yy}^n(\cdot,\cdot,t) - \varphi_{yy}(\cdot,\cdot,t)
		\right\rVert_{H^{-2}(\Omega)}\\
		& \leq c 
		\left\lVert
		\varphi^n(\cdot,\cdot,t) - \varphi(\cdot,\cdot,t)
		\right\rVert_{L^2(\Omega)} \leq cL^2
		\left\lVert
		\varphi_x^n(\cdot,\cdot,t) - \varphi_x(\cdot,\cdot,t)
		\right\rVert_{L^2(\Omega)}\\&= cL^2 
		\left\lVert
		v^n(\cdot,\cdot,t) - v(\cdot,\cdot,t)
		\right\rVert_{L^2(\Omega)}.
	\end{split}
\end{equation*}
Therefore, the desired convergence follows from the previous inequality and convergence \eqref{sim}.

 Therefore, from the above convergences $v(x,y,t)$ satisfies \eqref{v_st} and $\eqref{eq:MU}$ with the following conditions
\begin{equation}\label{eq:MUv}
\begin{cases}
	v(0,y,t)=v(L,y,t)= 0,& y\in(0,L),\  t>0\\
	\partial_xv(L,y,t) = \partial_xv(0,y,t)=\partial^2_xv(L,y,t)=0,& y\in(0,L),\  t>0\\
	v(x,L,t)=v(x,0,t)=0,& x\in (0,L),\ t>0\\
	\left\lVert v \right\rVert_{L^2(0,T, L^2(\Omega))}=1.
\end{cases}
\end{equation}
Thus,  for $\lambda=0$ we obtain $v=0$,  thanks to Holmgren's uniqueness theorem, which is a contradiction with the fact that $\left\lVert v \right\rVert_{L^2(0,T, L^2(\Omega))} = 1$.  Otherwise,  if $\lambda>0$,  we can show that $v\in L^2(0,T, H_x^5(\Omega)\cap X^2(\Omega))$ and applying \cite[Theorem 1.2]{ailton2021}, follows that $u\equiv 0$ in $\Omega\times (0,T)$,  achieving Theorem \ref{th:MUes}. 
\qed

\section*{Acknowledgment}

Capistrano--Filho was supported by CAPES grant 88881.311964/2018-01 and \linebreak 88881.520205/2020-01,  CNPq grant 307808/2021-1 and  401003/2022-1,  MATHAMSUD grant 21-MATH-03 and Propesqi (UFPE).  Galeano acknowledges support from FACEPE grant IBPG-0909-1.01/20.  This work is part of the Ph.D. thesis of Mu\~{n}oz at the Department of Mathematics of the Universidade Federal de Pernambuco.

\appendix

\section{\texorpdfstring{$\mu_i$}{}-system: Well-posedness}\label{Sec3} In this appendix, we deal with the study of the $\mu_i$-system \eqref{eq:MU} that is essential to obtain results for \eqref{eq:KP}. Since the results are classical, we just give the main results and the idea of the proofs.

\subsection{Linear system} Here,  we use semigroup theory to obtain well-posedness results for the linear system associated with \eqref{eq:MU}.  To do that,  consider $z(x,y,\rho,t)=u(x,y,t-\rho h)$, for $(x,y)\in \Omega$, $\rho\in (0,1)$ and $t>0$. Then $z(x,y,\rho,t)$ satisfies the transport equation
\begin{equation}\label{eq:MUtp}
\begin{cases}
h\partial_tz(x,y,\rho,t)+\partial_{\rho}z(x,y,\rho,t)=0, & (x,y)\in \Omega,\ \rho\in (0,1),\ t>0, \\
z(x,y,0,t)=u(x,y,t), & (x,y) \in\Omega,\ t>0, \\
z(x,y,\rho,0) = z_0(x,y,\rho,-\rho h), & (x,y)\in \Omega, \ \rho\in(0,1).
\end{cases}
\end{equation}

Let $\mathcal{H}=L^2(\Omega)\times L^2\left(\Omega\times(0,1)\right)$ a Hilbert space equipped with the inner product
\begin{equation*}
\begin{split}
	\left\langle
		(u,z)\,(v,w)
	\right\rangle_{\mathcal{H}}
=&\int_0^L\int_0^L  u(x,y)v(x,y)\,dx\,dy\\&+\xi\left\lVert a\right\rVert_\infty \int_0^L\int_0^L  \int_0^1 z(x,y,\rho)w(x,y,\rho)\, d\rho\,dx\,dy,
\end{split}
\end{equation*}
with $\xi$ satisfies~\eqref{eq:MUcond}. To study the well-posedness in the Hadamard sense, we need to rewrite the linear system associated with \eqref{eq:MU} as an abstract problem. Let $U(t)=\left(u(\cdot,\cdot,t), z(\cdot,\cdot,\cdot,t)\right)$ and denote $z(1):=z(x,y,1,t)$. From the linear system associated with \eqref{eq:MU} and~\eqref{eq:MUtp} we get the next system
\begin{equation}\label{eq:MUlin1}
\begin{split}
\begin{cases}
	\begin{aligned} 
	&\partial_tu(x,y,t) + \alpha \partial_{x}^3u(x,y,t) + \beta\partial_{x}^5u(x,y,t) \\ &+ \gamma\partial_x^{-1} \partial^2_yu(x,y,t)+ \frac{1}{2}\partial_x(u^2(x,y,t)) \\
&+ a(x,y)\left(\mu_1u(x,y,t) + \mu_2 z(1)\right)=0
	\end{aligned}& (x,y,t)\in \Omega\times\mathbb{R}^+\\
	u(0,y,t)=u(L,y,t)=0,& y\in(0,L), \ t\in(0,T), \\
	\partial_x u(L,y,t) =\partial_x u(L,y,t) =0,&y\in(0,L), \ t\in(0,T), \\
    \partial_xu(0,y,t)=\partial^2_xu(L,y,t)=0,&y\in(0,L), \ t\in(0,T), \\
	u(x,L,t)=u(x,0,t)=0,& x\in(0,L), \ t\in(0,T),\\
	u(x,y,0)=u_0(x,y), & (x,y)\in \Omega \\
h\partial_tz(x,y,\rho,t)+\partial_{\rho}z(x,y,\rho,t)=0, & (x,y)\in \Omega,\ \rho\in (0,1),\ t>0, \\
z(x,y,0,t)=u(x,y,t), & (x,y) \in\Omega,\ t>0, \\
z(x,y,\rho,0) = z_0(x,y,\rho,-\rho h), & (x,y)\in \Omega, \ \rho\in(0,1).
\end{cases}
\end{split}
\end{equation}
which is equivalent to
\begin{equation}\label{eq:MUabs}
\begin{cases}
\dfrac{d}{d t}U(t)= AU(t), \\
U(0)=\left(u_0(x,y), z_0(x,y,-\rho h)\right)
\end{cases}
\end{equation}
where $A\colon D(A)\subset \mathcal{H}\to \mathcal{H}$ is defined by
\begin{equation}\label{eq:MUA}
A(u, z) = 
	\left(%
		-\alpha \partial^3_xu-\beta\partial^5_x u-\gamma {\partial_x^{-1}} \partial^2_yu-a(x,y)(\mu_1 u+\mu_2 z(1)); -h^{-1}\partial_{\rho}z\right)
\end{equation}
with the dense domain given by
\begin{equation*}
D(A):=
	\left\lbrace
		\begin{aligned}
			&(u,z)\in \mathcal{H} \colon \\
			&u\in H_x^5(\Omega)\cap X^2(\Omega),\\
			&\partial_{\rho}z\in L^2(\Omega\times(0,1)),
		\end{aligned} \left\vert
		\begin{aligned}
			&u(0,y)=u(L,y)=u(x,0)=u(x,L)=0,\\
			&\partial_xu(L,y)=\partial_xu(0,y)=\partial^2_xu(L,y)=0, \\
			&z(x,y,0)=u(x,y) 
			\end{aligned}\right.
	\right\rbrace.
\end{equation*}

The next result is classical and can be omitted. 

\begin{lemma}\label{le:MUAadj}
The operator $A$ is closed and the adjoint $A^\ast\colon D(A^\ast)\subset \mathcal{H}\to \mathcal{H}$ is given by
\begin{equation*}
A^\ast(u,z)=
	\left(%
		\alpha \partial^3_xu+\beta \partial^5_xu+\gamma {\partial_x^{-1}} \partial^2_yu-a(x,y)\mu_1 u+\frac{\xi\left\lVert a\right\rVert_\infty}{h}z(\cdot,\cdot,0); h^{-1}\partial_{\rho}z\right)
\end{equation*}
with dense domain
\begin{equation*}
D(A^\ast):=
	\left\lbrace
		\begin{aligned}
			&(u,z)\in \mathcal{H} \colon \\
			&u\in H_x^5(\Omega)\cap X^2(\Omega),\\
			&\partial_{\rho}z\in L^2(\Omega\times(0,1)),
		\end{aligned} \left\vert
		\begin{aligned}
			&u(0,y)=u(L,y)=u(x,0)=u(x,L)=0,\\
			&\partial_xu(L,y)=\partial_xu_x(0,y)=\partial^2_xu(0,y)=0, \\
			&z(x,y,1)=-\dfrac{a(x,y)h\mu_2}{\xi\left\lVert a\right\rVert_{\infty}}u(x,y) 
			\end{aligned}\right.
	\right\rbrace.
\end{equation*}
\end{lemma}
\begin{proposition}\label{pr:MUdis}
Assume that $a\in L^\infty(\Omega)$ is a nonnegative function and~\eqref{eq:MUcond} is satisfied. Then $A$ is the infinitesimal generator of a $C_0$-semigroup in $\mathcal{H}$.
\end{proposition}
\begin{proof}
Let $U=(u,z)\in D(A)$, then
\begin{equation}\label{eq:MUAdis}
	\left\langle AU,U\right\rangle_\mathcal{H}\leq \frac{\xi\left\lVert{a}\right\rVert_\infty}{2h}\int_0^L\int_0^L  u^2(x,y)\,dx\,dy.
\end{equation}
Hence, for $\lambda=\frac{\xi\left\lVert{a}\right\rVert_\infty}{2h}$ we have $\left\langle(A-\lambda I)U,U\right\rangle_\mathcal{H}\leq 0$ (resp. $\left\langle(A^\ast-\lambda I)U,U\right\rangle_\mathcal{H} \leq 0,$ for $U=(u,z)\in D(A^\ast)$).  Since $A-\lambda I$ is a densely defined closed linear operator, and both $A-\lambda I$ and $(A-\lambda I)^\ast$ are dissipative,  $A$ generate an infinitesimal $C_0$-semigroup on $\mathcal{H}$.
\end{proof}

The next theorem establishes the existence of solutions for the abstract Cauchy problem~\eqref{eq:MUabs}. This result is a consequence of the previous proposition. 

\begin{theorem}
Assume that  $a\in L^\infty(\Omega)$ and~\eqref{eq:MUcond} is satisfied. Then, for each initial data $U_0\in \mathcal{H}$ there exists a unique mild solution $U\in C\left([0,\infty), \mathcal{H}\right)$ for the system~\eqref{eq:MUabs}. Moreover, if the initial data $U_0\in D(A)$ the solutions are classical such that $
U\in C\left([0,\infty), D(A)\right)\cap C^1\left([0,\infty), \mathcal{H}
\right).$
\end{theorem}

Next results are devoted to showing \textit{a priori} and regularity estimates for the solutions of \eqref{eq:MUabs}.

\begin{proposition}
Let $a\in L^\infty(\Omega)$ be a nonnegative function and consider that \eqref{eq:MUcond} holds. Then, for any mild solution of~\eqref{eq:MUabs} the energy $E_u$, defined by~\eqref{eq:MUen}, is non-increasing and there exists a constant $C>0$ such that
\begin{equation}\label{eq:MUendis}
\begin{split}
\frac{d}{dt}E_u(t) \leq& -C
	\left(
		\int_0^L \partial^2_xu(0,y,t)^2\,dy+ \int_0^L({\partial_x^{-1}} \partial_yu(0,y,t))^2\,dy
	\right.\\
&\left.+\int_0^L\int_0^L a(x,y)u^2\,dx\,dy
		+\int_0^L\int_0^L a(x,y)u^2(x,y,t-h)\,dx\,dy%
	\right)
\end{split}
\end{equation}
where $C=C(\beta,\gamma,\xi,h,\mu_1,\mu_2)$ is given by
\begin{equation*}
C=\min
	\left\lbrace%
		-\frac{\beta}{2},\frac{\gamma}{2}, \mu_1-\frac{\mu_2}{2}-\frac{\xi}{2h}, -\frac{\mu_2}{h}+\frac{\xi}{2h}
	\right\rbrace.
\end{equation*}
\end{proposition}
\begin{proof}
First, multiply \eqref{eq:MUlin1}$_1$ by $u(x,y,t)$ and integrate by parts in $L^2(\Omega)$. After that, multiply \eqref{eq:MUlin1}$_5$ by $z(x,y,\rho,t)$ and integrate by parts in $L^2(\Omega\times(0,1))$.  Finally, adding the results, and the proposition follows.
\end{proof}

To use the contraction principle and to obtain the Kato smoothing effect, for $T>0$, we introduce the following sets:
\begin{equation*}
\begin{split}
\mathcal{B}_X
= &C\left([0,T], L^2(\Omega)\right)\cap L^2\left(0,T ,X_{x0}^2 (\Omega)\right), \\
\mathcal{B}_H= &C\left([0,T], L^2(\Omega)\right)\cap L^2\left(0,T ,H_{x0}^2 (\Omega)\right)
\end{split}
\end{equation*}
endowed with its natural norms
\begin{equation*}
\begin{aligned}
\left\lVert{y}\right\rVert_{\mathcal{B}_X}& 
= \max_{t\in[0,T]}\left\lVert{y(\cdot,\cdot,t)}\right\rVert_{L^2(\Omega)}+\left(\int_0^T \left\lVert{y(\cdot,\cdot, t)}\right\rVert_{X_{x0}^2(\Omega)}^2\,dt \right)^{\frac{1}{2}}, \\
\left\lVert{y}\right\rVert_{\mathcal{B}_H}&
= \max_{t\in[0,T]}\left\lVert{y(\cdot,\cdot,t)}\right\rVert_{L^2(\Omega)}+\left(\int_0^T \left\lVert{y(\cdot,\cdot, t)}\right\rVert_{H_{x0}^2(\Omega)}^2\,dt \right)^{\frac{1}{2}}.
\end{aligned}
\end{equation*}
Here, $X_{x0}^2(\Omega)$ denotes the space
\begin{equation}\label{eq:Xx0k}
X_{x0}^{k}(\Omega):=
\left\{
\begin{array}
[c]{l}
\varphi \in H_{x0}^k(\Omega) \colon {\partial_x^{-1}} \varphi(x,y)= \psi(x,y)\in H_{x0}^k(\Omega)\text{ with }\\\psi(L,y)=0\text{ and }\partial_x\psi(x,y)=\varphi(x,y).
\end{array}
\right\}
\end{equation}

\begin{proposition}\label{pr:MUreg}
Let $a\in L^\infty(\Omega)$ be a nonnegative function. Then, the map
$$
\left(u_0, z_0(\cdot,\cdot,-h(\cdot))\right)\in \mathcal{H}
\mapsto 
\left(u,z\right) \in
\mathcal{B}_X\times C\left([0,T], L^2\left(\Omega\times(0,1)\right)\right)
$$
is continuous and for $\left(u_0, z_0(\cdot,\cdot,-h(\cdot))\right)\in \mathcal{H}$, the following estimates are satisfied
\begin{equation}\label{eq:MUr1}
\begin{split}
&\frac{1}{2} \int_0^L\int_0^L  u^2(x,y)\,dx\,dy 
+ \frac{\xi}{2}\int_0^L\int_0^L \int_0^1 a(x,y) u^2(x,y,t-\rho h)\,d\rho\,dx\,dy 
\\ &\leq \frac{1}{2} \int_0^L\int_0^L  u_0^2\,dx\,dy 
+\frac{\xi}{2}\int_0^L\int_0^L \int_0^1 a(x,y)z_0^2(x,y,-\rho h)\,d\rho\,dx\,dy,
\end{split}
\end{equation}
\begin{equation}\label{eq:MUr2}
\begin{split}
\frac{3\alpha}{2}\int_0^T\int_0^L\int_0^L  \partial_x u(x,y,t)^2\,dx\,dy\,dt
-\frac{5\beta}{2}\int_0^T\int_0^L\int_0^L  \partial^2_xu(x,y,t)^2\,dx\,dy\,dt \\
\leq \mathcal{C}(a,\mu_1,\mu_2, L)(1+T)
	\left\lVert%
		{(u_0, z_0(\cdot,\cdot, -h(\cdot)))}
	\right\rVert_{\mathcal{H}}
	\end{split}
\end{equation}
and
\begin{equation}\label{eq:MUr3}
\begin{split}
&\left\lVert{u_0}\right\rVert_{L^2(\Omega)}^2 
\leq \frac{1}{T} \int_0^T\int_0^L\int_0^L  u^2(x,y,t)\,dx\,dy\,dt 
- \beta \int_0^T\int_0^L \partial^2_xu(0,y,t)^2\,dy\,dt 
\\&+ \gamma \int_0^T\int_0^L 
	\left(%
		{{\partial_x^{-1}} \partial_yu(0,y,t)}
	\right)^2\,dy\,dt+ \int_0^T\int_0^L\int_0^L  a(x,y) \mu_2 u^2(x,y,t-h)\,dx\,dy\,dt
\\&+(2\mu_1+\mu_2) \int_0^T\int_0^L\int_0^L  a(x,y) u^2(x,y,t)\,dx\,dy\,dt .
\end{split}
\end{equation}
\end{proposition}

\begin{proof}
The proof is classical and uses the Morawetz multipliers.  Precisely,  first, \eqref{eq:MUr1} follows from \eqref{eq:MUendis}.  To get the other two inequalities for  $\left(u_0 , z_0(\cdot,\cdot, -h(\cdot))\right)\in\mathcal{H}$, multiplying~\eqref{eq:MUlin1}$_{5}$ by $z(x,\rho,t)$ and \eqref{eq:MUlin1}$_{1}$ by $xu(x,y,t)$ and integrating by parts in $\Omega\times (0,T)$, \eqref{eq:MUr2} holds. Finally,  multiplying \eqref{eq:MUlin1}$_{1}$ by $(T-t)u(x,y,t)$ and integrating by parts in $\Omega\times (0,T)$ we obtain \eqref{eq:MUr3}.
\end{proof}

\subsection{Linear system with source term}
We will study the system \eqref{eq:MUlin1}, with a source term $f(x,y,t)$ on the right-hand side. The next result ensures the well-posedness of this system.
\begin{proposition}
	Assume that  $a(x,y)\in L^\infty(\Omega)$ is a nonnegative function and that \eqref{eq:MUcond} is satisfied. For any $\left(u_0,z_0(\cdot,\cdot, -h(\cdot))\right)\in \mathcal{H}$ and $f\in L^1\left(0,T,L^2(\Omega)\right)$, there exists a unique mild solution for \eqref{eq:MUlin1} with the source term $f(x,y,t)$ on the right-hand side in the class
$$
\left(u,u(\cdot,\cdot, t-h(\cdot))\right)\in \mathcal{B}_X\times C\left([0,T], L^2(\Omega\times(0,1))\right).
$$
Moreover, we have
\begin{equation}\label{eq:MUsou1}
\left\lVert{(u,z)}\right\rVert_{C([0,T],\mathcal{H})}\leq e^{\frac{\xi\left\lVert{a}\right\rVert_\infty}{2h}T}
	\left(%
		\left\lVert{(u_0,z_0(\cdot,\cdot,-h(\cdot)))}\right\rVert_{\mathcal{H}} 
		+\left\lVert{f}\right\rVert_{L^1(0,T,L^2(\Omega))}
	\right)
\end{equation}
and
\begin{equation}\label{eq:MUsou2}
\delta\left\lVert{u}\right\rVert_{L^2(0,T,H_x^2(\Omega))}^2
\leq \mathcal{C}
	\left(%
		\left\lVert (u_0,z_0 (\cdot,\cdot,-h(\cdot)))\right\rVert_{\mathcal{H}}^2
		+\left\lVert{f}\right\rVert_{L^1(0,T,L^2(\Omega))}^2
	\right)
\end{equation}
where  
\begin{equation*}
\mathcal{C} = \mathcal{C}
	\left(a,\mu_1,\mu_2,L,T,h\right)=\frac{3L}{2}+L\left\lVert{a}\right\rVert_\infty(\mu_1+\mu_2)+\delta\left(1+T+e^{\frac{\xi\left\lVert{a}\right\rVert_\infty}{h}T}\right)
\end{equation*}
and $\delta=\min\left\lbrace{1,3\alpha/2, -5\beta/{2}}\right\rbrace$.
\end{proposition}

\begin{proof}
Note that $A$ is an infinitesimal generator of a $C_0$-semigroup $(e^{tA})_{t\geq0}$ satisfying $\left\lVert{e^{tA}}\right\rVert_{\mathcal{L}(\mathcal{H})}\leq e^{\frac{\xi\left\lVert{a}\right\rVert_\infty}{2h}t}$ and the system can be rewritten as a first order system with source term $(f(\cdot,\cdot,t),0)$, showing the well-posed in $C([0,T],\mathcal{H})$.  Finally, observe that the right-hand side is not homogeneous, since
\begin{equation*}
	\left\lvert%
		\int_0^T\int_{\Omega} xf(x,y,t) u(x,y,t)\,dx\,dy\,dt
	\right\rvert%
\leq \frac{L}{2} \left\lVert{u}\right\rVert_{C([0,T]; L^2(0,L))}^2+ \frac{L}{2}\left\lVert{f}\right\rVert_{L^1(0,T, L^2(\Omega))}^2,
\end{equation*}
showing the result.
\end{proof}

\subsection{Nonlinear system: Global results.}\label{ss:MUnl} In this last part, we consider the nonlinear term $uu_x$ as a source term.
\begin{proposition}\label{pr:NLcont}
If $u\in \mathcal{B}_X$ then $uu_x\in L^1(0,T; L^2(\Omega))$ and the map
$
u\in\mathcal{B}_X \mapsto u\partial_xu\in L^1(0,T; L^2(\Omega))
$
is continuous. In particular, exists $K>0$,  such that, for all $u,v\in \mathcal{B}_X$ we have
\begin{equation}\label{eq:MUnl}
\left\lVert{u\partial_xu-v\partial_xv}\right\rVert_{L^1(0,T,L^2(\Omega))}
\leq K
	\left(%
		\left\lVert{u}\right\rVert_{\mathcal{B}_X}
		+\left\lVert{v}\right\rVert_{\mathcal{B}_X}
	\right)\left\lVert{u-v}\right\rVert_{\mathcal{B}_X}.
\end{equation}
\end{proposition}
\begin{proof}
The Hölder inequality and the Sobolev embedding $H_{x0}^2(\Omega)\hookrightarrow L^\infty(\Omega)$ gives us
\begin{equation}\label{eq:MUnl1}
\left\lVert u\partial_xu-v\partial_xv \right\rVert_{L^1(0,T, L^2(\Omega))}
 \leq C_1\cdot C \cdot T^{\frac{1}{4}}\left(\left\lVert{u}\right\rVert_{\mathcal{B}_{H}} + \left\lVert{v}\right\rVert_{\mathcal{B}_{H}}
	\right)\left\lVert{u-v}\right\rVert_{\mathcal{B}_H},
\end{equation}
for $u,v\in \mathcal{B}_X$.  Note that, $u\in\mathcal{B}_X$ implies that 
$u(\cdot,\cdot, t)\in H_{x0}^2(\Omega)$
and consequently 
$u(\cdot,\cdot,t)\in H_{x0}^1(\Omega)$ and
$u_x(\cdot,\cdot,t)\in H_{x0}^1(\Omega)$. 
Here, using the definition of the operator ${\partial_x^{-1}}$ and the 
Poincaré's inequality we obtain,
\begin{equation}\label{eq:MUnl2}
\left\lVert
	{\partial_x^{-1}}(u\partial_xu)
\right\rVert_{L^1(0,T, L^2(\Omega))} \leq L^2
\left\lVert
	u\partial_xu
\right\rVert_{L^1(0,T, L^2(\Omega))}
\end{equation}
So,  from~\eqref{eq:MUnl1}, with $v=0$,  and~\eqref{eq:MUnl2} we get $u\partial_xu \in L^1(0,T, L^2(\Omega))$ and the proof is complete.
\end{proof}

We prove the global well-posedness of the K-KP-II with delay term.
\begin{proposition}\label{pr:MUkato}
Let $L>0$, $a(x,y)\in L^\infty(\Omega)$ be a nonnegative function and that \eqref{eq:MUcond} holds. Then, for all initial data $(u_0,z_0(\cdot,\cdot, -h(\cdot))\in \mathcal{H}$, there exists a unique $u\in \mathcal{B}_X$ solution of~\eqref{eq:MU}. Moreover, there exist constants $\mathcal{C}>0$ and $\delta\in(0,1]$ such that 
\begin{equation}\label{eq:MUkato}
\delta\left\lVert u \right\rVert_{L^2(0,T, H_x^2(\Omega))}^2
\leq \mathcal{C}\left(%
	\left\lVert (u_0,z_0(\cdot,\cdot,-h(\cdot))) \right\rVert_{\mathcal{H}}^2
+
	\left\lVert (u_0,z_0(\cdot,\cdot,-h(\cdot))) \right\rVert_{\mathcal{H}}^{\frac{10}{3}}
	\right).
\end{equation}
\end{proposition}

\begin{proof}
To obtain the global existence of solutions we show the local existence and use the \textit{a priori} estimate below, which is proved using the multipliers method and Gronwall's inequality:
\begin{equation}\label{eq:MUkato1}
\left\lVert{(u(\cdot,\cdot,t),u(\cdot,\cdot,t-h))}\right\rVert_{\mathcal{H}}^2 
\leq e^{\frac{\xi\left\lVert{a}\right\rVert_\infty}{h}t}
\left\lVert{(u_0,z_0(\cdot,\cdot,-h(\cdot)))}\right\rVert_{\mathcal{H}}^2
\end{equation}

With the previous inequality in hands, the local existence and uniqueness of solutions of~\eqref{eq:MU} holds.  Precisely,  pick $(u_0,z_0(\cdot,\cdot,-h(\cdot)))\in \mathcal{H}$ and $u\in \mathcal{B}_X$, consider the map $\Phi\colon \mathcal{B}_X\to \mathcal{B}_X$ defined by $\Phi(u)=\tilde{u}$, where $\tilde{u}$ is solution of \eqref{eq:MU} with the source term $f=-u\partial_xu$.  Then, $u\in \mathcal{B}_X$ is the solution for~\eqref{eq:MU} if and only if $u$ is a fixed point of $\Phi$.  To show this, we need to prove that $\Phi$ is a contraction.  

If $T<1$ from~\eqref{eq:MUsou1},~\eqref{eq:MUsou2} and Proposition~\ref{pr:NLcont} we get
\begin{equation*}
\begin{aligned}
\left\lVert \Phi u \right\rVert_{\mathcal{B}_X}
 \leq & \sqrt{\delta^{-1}\mathcal{C}}
	\left(%
		1+\sqrt{T}+e^{\frac{\xi\left\lVert{a}\right\rVert_\infty}{2h}T}
	\right)%
		\left\lVert (u_0,z_0 (\cdot,\cdot,-h(\cdot)))\right\rVert_{\mathcal{H}}\\
& + \sqrt{\delta^{-1}\mathcal{C}}\cdot C_1\cdot C
	\left(%
		2T^{\frac{1}{4}} + T^{\frac{1}{4}} e^{\frac{\xi\left\lVert{a}\right\rVert_{\infty}}{2h}T}
	\right)%
		\left\lVert{u}\right\rVert_{\mathcal{B}_X}^2
\end{aligned}
\end{equation*}
and
\begin{equation*}
\left\lVert{ \Phi{u}-\Phi{v} }\right\rVert_{\mathcal{B}_X} 
\leq  S
	\left(%
		1+\sqrt{T}+e^{\frac{\xi\left\lVert{a}\right\rVert_\infty}{2h}T}
	\right) T^{\frac{1}{4}} \left(
		\left\lVert{u}\right\rVert_{\mathcal{B}_{X}} + \left\lVert{v}\right\rVert_{\mathcal{B}_{X}}
	\right)\left\lVert{u-v}\right\rVert_{\mathcal{B}_X},
\end{equation*}
where $S=\sqrt{\delta^{-1}\mathcal{C}}\cdot C_1\cdot C$. Now,  consider the application $\Phi$ restricted to the closed ball
$$
\left\lbrace{ %
	u\in \mathcal{B} \colon \left\lVert{u}\right\rVert_{\mathcal{B}_X}\leq R
}\right\rbrace,
$$
with $R>0$ such that $R=4\sqrt{\delta^{-1} \mathcal{C}} \left\lVert (u_0, z_0(\cdot,\cdot,-h(\cdot)) \right\rVert_{\mathcal{H}}$ and $T>0$ satisfying
\begin{equation*}
T<1,\quad e^{\frac{\xi\left\lVert{a}\right\rVert_{\infty}}{2h}T} <2,\quad 2T^{\frac{1}{4}}+T^{\frac{1}{4}}e^{\frac{\xi\left\lVert{a}\right\rVert_{\infty}}{2h}T}< \frac{1}{2\sqrt{\delta^{-1}\mathcal{C}}\cdot C_1\cdot C_2 R}
\end{equation*}
holds that $\Phi$ is a contraction.  From Banach's fixed point theorem,  application $\Phi$ has a unique fixed point.
\end{proof}

\end{document}